\documentclass[10pt,letterpaper]{article}
\usepackage[utf8]{inputenc}
\usepackage{amsmath}
\usepackage{amsfonts}
\usepackage{amssymb}
\usepackage{graphicx}
\usepackage{mathrsfs}
\usepackage{upref,amsthm,amsxtra,exscale}
\usepackage{cite}
\usepackage[colorlinks=true,urlcolor=blue,
citecolor=red,linkcolor=blue,linktocpage,pdfpagelabels,
bookmarksnumbered,bookmarksopen]{hyperref}
\usepackage{upgreek}

\usepackage{fullpage} 
\usepackage{soul}

\newtheorem{theorem}{Theorem}[section]
\newtheorem{corollary}[theorem]{Corollary}
\newtheorem{remark}[theorem]{Remark}
\newtheorem{lemma}[theorem]{Lemma}
\newtheorem{proposition}[theorem]{Proposition}

\numberwithin{equation}{section}

\def\r{\mathbb{R}}

\def\rn{\mathbb{R}^N}

\def\n{\mathbb{N}}

\def\eps{\varepsilon}
\def\rh{\rightharpoonup}

\def\irn{\int_{\r^N}}
\def\vp{\varphi}

\def\o{\Omega}

\def\bf{\boldsymbol}

\def\cC{\mathcal{C}}

\def\cH{\mathcal{H}}

\def\cJ{\mathcal{J}}

\def\cM{\mathcal{M}}
\def\cN{\mathcal{N}}

\def\cT{\mathcal{T}}
\def\cU{\mathcal{U}}

\def\supp{\mathrm{supp}}

\def\what{\widehat}
\def\gen{\mathrm{genus}}

\def\r{\mathbb{R}}
\def\rn{\mathbb{R}^N}

\def\n{\mathbb{N}}

\def\eps{\varepsilon}
\def\rh{\rightharpoonup}

\def\irn{\int_{\r^N}}

\def\vp{\varphi}

\def\cJ{\mathcal{J}}
\def\cM{\mathcal{M}}
\def\cN{\mathcal{N}}

\def\cU{\mathcal{U}}

\def\bt{\mathbf{t}}

\author{Mónica Clapp,
Alberto Saldaña\footnote{ A. Saldaña is supported  by
CONAHCYT grant CBF2023-2024-116 (Mexico) and by UNAM-DGAPA-PAPIIT grant IA100923 (Mexico).
}
\ and Andrzej Szulkin}
\title{On a Schrödinger system with shrinking regions of attraction}
\date{}
        
\begin{document}
\maketitle

\begin{abstract}
In this paper we consider a competitive weakly coupled elliptic system in which each species is attracted to a small region in $\rn$ and repelled from its complement. In this setting, we establish the existence of infinitely many solutions and of a nonnegative least energy solution. We show that, as the regions of attraction shrink, least energy solutions of the system concentrate. We study this behavior and characterize their limit profile. In particular, we show that if each component of a least energy solution is attracted to a different region, then the components decouple in the limit, whereas if all the components are attracted to the same region, they remain coupled.

\smallskip

\emph{Key words and phrases.} Shrinking regions of attraction, concentration, limit profile, competitive weakly coupled elliptic system, phase separation.
\smallskip

\emph{2020 Mathematics Subject Classification.}
35J50, 
35B08, 
35B20, 
35B40. 

\end{abstract}

\section{Introduction}

Consider the system of elliptic equations
\begin{equation} \label{eq:system} 
\begin{cases}
-\Delta v_i + v_i = \mu_iQ_\varepsilon(x-y_i)|v_i|^{2p-2}v_i + \sum\limits_{\substack{j=1 \\ j\neq i}}^\ell\lambda_{ij}|v_j|^p|v_i|^{p-2}v_i, \\
v_i\in H^1(\rn), \quad v_i\ne 0,\quad i=1,\ldots,\ell,
\end{cases}
\end{equation}
where $N\ge 3$, $y_i\in\rn$, $\mu_i>0$, $\lambda_{ij}=\lambda_{ji}<0$, $1<p<2^*/2$, $\eps>0$ and
\begin{equation}\label{Qdef}
Q_{\eps}(x):=
\begin{cases}
1 &\text{if \ }|x|<\eps, \\
-1 &\text{if \ }|x|\geq\eps.
\end{cases}
\end{equation}
As usual, $2^*$ is the critical Sobolev exponent, i.e., $2^* := \frac{2N}{N-2}$.  In this context, the solutions $v_i$ may represent some particles or species that are attracted to the small region $B_\eps(y_i)=\{x\in \rn\::\: |x-y_i|<\eps\}$ and repelled from its complement. Furthermore, since $\lambda_{ij}<0$, the system is \emph{competitive}, and therefore different species repel each other.

In the present paper we are interested in studying the existence and the limit profile of solutions to~\eqref{eq:system} as $\eps\to 0$.

The following is our main existence result.  A solution is called \emph{nonnegative} if $v_i\ge 0$ for all $i$.

\begin{theorem} \label{thm:multiplicity}
The system~\eqref{eq:system} has an unbounded sequence of solutions. At least one of them is a nonnegative least energy solution.
\end{theorem}

Next we analyse the asymptotic behavior of least energy solutions as the parameter $\eps$ tends to zero, which corresponds to shrinking of the regions of attraction.  We show that this leads to different concentration phenomena depending on whether the points $y_i$ are distinct or not. We study the cases where all the points are distinct ($y_i\neq y_j$ for $i\neq j$) and when all points are the same ($y_i=y_j$ for all $i,j=1,\ldots,\ell$), and via a suitable rescaling and translation, we show that this leads to different limit profiles. Note that the system~\eqref{eq:system} does not possess a nontrivial solution for $\eps=0$ as can be easily seen by multiplying the equation in~\eqref{eq:system} by $v_i$ and integrating.

Let $E:=D^{1,2}(\rn)\cap L^{2p}(\rn)$ be the Banach space whose norm is given by
$$\|u\|_E^2:=\|u\|^2+|u|_{2p}^2,\qquad\text{where \ } \ \|u\|^2:=\irn|\nabla u|^2\text{ \ and \ }|u|_{2p}^{2p}:=\irn|u|^{2p}.$$
The following is our main result when all the concentration points $y_i$ are different.

\begin{theorem} \label{thm:uncoupled}
Assume that $y_i\neq y_j$ if $i\neq j$. Let $\eps_n\to 0$ and let $\bf v_n=(v_{n,1},\ldots,v_{n,\ell})$ be a nonnegative least energy solution to the system~\eqref{eq:system} with $\eps=\eps_n$. Set $w_{n,i}(x):=\eps_n^{\frac1{p-1}} v_{n,i}(\eps_n x + y_i)$. Then, after passing to a subsequence, $w_{n,i}\to w_i$ strongly in $E$ to a positive least energy solution $w_i$ of the equation
\begin{equation}\label{eq:limiting}
\begin{cases}
-\Delta w = \mu_iQ_1(x)|w|^{2p-2}w, \\
w\in E:=D^{1,2}(\rn)\cap L^{2p}(\rn), \quad w\neq 0,
\end{cases}
\end{equation}
for every $i=1,\ldots,\ell$. As a consequence, for any $\delta>0$, 
\begin{align}\label{q}
 \lim_{n\to\infty} \frac{\int_{|x-y_i|\leq\delta}(|\nabla v_{n,i}|^2+v_{n,i}^2)}{\irn (|\nabla v_{n,i}|^2+v_{n,i}^2)} = 1 \qquad \text{and} \qquad \lim_{n\to\infty} \frac{\int_{|x-y_i|\leq\delta}|v_{n,i}|^{2p}}{\irn |v_{n,i}|^{2p}} = 1.
\end{align}
\end{theorem}

The concentration is reflected in the limits in~\eqref{q}. Theorem~\ref{thm:uncoupled} shows that, when all the concentration points are different, after rescaling and translation, the components \emph{decouple} as $\eps\to 0$ and the limit profile of the $i$-th component is a least energy solution of \emph{the single equation}~\eqref{eq:limiting}. This equation has been studied in \cite{chs}, where it is shown that any positive least energy solution $w$ of~\eqref{eq:limiting} is radially symmetric with respect to the origin and decreasing in the radial direction. Moreover, for $N\ge 3$ there is $C>0$ such that
$w(x)\le C|x|^{2-N}$. A lower bound for $w$ is also available if $N\ge 3$ and $p\in(\frac{N-1}{N-2},\frac{N}{N-2})$. See \cite[Theorems 1.4 and 1.5]{chs} for these results.

The next theorem studies the case when all the concentration points $y_i$ are the same.

\begin{theorem} \label{thm:coupled}
Assume that $y_i=0$ for every $i=1,\ldots,\ell$. Let $\eps_n\to 0$ and let $\bf v_n=(v_{n,1},\ldots,v_{n,\ell})$ be a nonnegative least energy solution to the system~\eqref{eq:system} with $\eps=\eps_n$. Set $u_{n,i}(x):=\eps_n^{\frac1{p-1}} v_{n,i}(\eps_n x)$. Then, after passing to a subsequence, $(\bf u_n)$ converges strongly in $E$ to a nonnegative least energy solution of the limit system
\begin{equation} \label{eq:system3} 
\begin{cases}
-\Delta u_i = \mu_iQ_1(x)|u_i|^{2p-2}u_i + \sum\limits_{\substack{j=1 \\ j\neq i}}^\ell\lambda_{ij}|u_j|^p|u_i|^{p-2}u_i, \\
u_i\in E:=D^{1,2}(\rn)\cap L^{2p}(\rn),\quad u_i\neq 0,\quad i=1,\ldots,\ell.
\end{cases}
\end{equation}
As a consequence, for any $\delta>0$, 
$$\lim_{n\to\infty} \frac{\int_{|x|\leq\delta}(|\nabla v_{n,i}|^2+v_{n,i}^2)}{\irn (|\nabla v_{n,i}|^2+v_{n,i}^2)} = 1 \qquad \text{and} \qquad \lim_{n\to\infty} \frac{\int_{|x|\leq\delta}|v_{n,i}|^{2p}}{\irn |v_{n,i}|^{2p}} = 1.$$
\end{theorem}

Here we see again that the solutions concentrate, but now the components remain  \emph{coupled} and the limit profile is a solution of \emph{the limit system}~\eqref{eq:system3}, i.e., the effect of the repelling forces $\lambda_{ij}$ remains.

Next, we analyse the effect of the repelling forces on the limit system~\eqref{eq:system3}. Similar autonomous systems have been widely studied and it is, for instance, well known that least energy solutions exhibit phase separation as the repelling forces $\lambda_{i,j}$ increase. We show that this is also true for the limit system~\eqref{eq:system3}. We shall write $B_r(x)$ for the open ball of radius $r$ and center at $x$.

\begin{theorem}\label{seg:thm}
For each $i,j=1,\ldots,\ell$, $i\neq j$, let $(\lambda_{k})$ be a sequence of negative numbers such that $\lambda_{k}\to -\infty$ as $k\to\infty$, and  let $\bf u_{k}=(u_{k,1},\ldots,u_{k,\ell})$ be a nonnegative least energy solution of the limit system~\eqref{eq:system3} with $\lambda_{ij}=\lambda_{k}$ for all $i\neq j$. Then, after passing to a subsequence, we have that $u_{k,i}\to u_{\infty,i}$ strongly in $E$,\, $u_{\infty,i}\geq 0$,\, $u_{\infty,i}u_{\infty,j}=0$ if $i\neq j$, \ $u_{\infty,i}$ is continuous in $\mathbb{R}^N$, and $u_{\infty,i}|_{\Omega_i}$ is a least energy solution to the problem
\begin{equation}\label{eq:omega_i}
\begin{cases}
-\Delta w = \mu_iQ_1(x)|w|^{2p-2}w,\\
w\in E_i:=D_0^{1,2}(\Omega_i)\cap L^{2p}(\Omega_i),\quad w\neq 0,
\end{cases}
\end{equation}
where $\Omega_i:=\{x\in\mathbb{R}^{N}:u_{\infty,i}(x)>0\}$ for each $i=1,\ldots,\ell$. Hence, $\o_i\cap\o_j=\emptyset$ if $i\neq j$ and $\o_i\cap B_1(0)\neq\emptyset$.
\end{theorem}

This type of segregation for other systems has been established in \cite{ctv}; see also the surveys \cite{sttz, t} and the references therein.

In the next result we show that concentration occurs for all solutions of~\eqref{eq:system} as $\eps_n\to 0$, not only for least energy ones.

\begin{theorem} \label{thm:concentr}
Let $\eps_n\to 0$ and let $\bf v_n=(v_{n,1},\ldots,v_{n,\ell})$ be a  solution to the system~\eqref{eq:system} with $\eps=\eps_n$. Then~\eqref{q} holds true for each $\delta>0$.
\end{theorem}

Theorems~\ref{thm:uncoupled},~\ref{thm:coupled} and~\ref{thm:concentr} extend to systems the results of Ackermann and Szulkin \cite{as} who showed that, for a single equation, the solutions concentrate at a point as $\eps\to 0$. A thorough discussion of the physical background in the setting of electromagnetic waves, and related references to literature in physics are also given in \cite{as}. In some of this literature also systems of equations have been discussed. The existence of a limit profile for the single equation, satisfying equation~\eqref{eq:limiting}, was established by Fang and Wang in \cite{fw}. In \cite{chs}, positive and nodal solutions to~\eqref{eq:limiting} are studied, including their symmetries and decay properties. We note that, for simplicity and to explain our main ideas in a transparent way, we have considered only a simple shape for the region of attraction (a ball, see~\eqref{Qdef}), but our arguments can be easily extended to consider more general regions as in \cite{chs}.

We also mention the paper \cite{liu}, where the case of a single equation with a \emph{sublinear} power ($2p\in (1,2)$) is studied. Moreover, a system of equations similar to ours has been considered in \cite{zz} and \cite{jsx}. In these two papers $\ell=2$, $\mu_i=0$ and $Q_\eps$ appears in front of the coupling terms. Also in this case, concentration occurs as $\eps\to 0$. 

There are several other systems that exhibit concentration. The study of the concentration behavior of solutions to the system of singularly perturbed elliptic equations
\begin{equation} \label{singular}
-\eps^2\Delta v_i + v_i = \mu_i|v_i|^{2p-2}v_i + \sum\limits_{\substack{j=1 \\ j\neq i}}^\ell\lambda_{ij }|v_j|^p|v_i|^{p-2}v_i, \qquad i= 1,\ldots,\ell,
\end{equation}
has aroused special interest, starting from the seminal paper by Lin and Wei \cite{lw}. In addition to its relevance in physics, understanding concentration has other interesting consequences. For example, it allows obtaining multiplicity of positive solutions to~\eqref{singular} in bounded domains, as recently shown in \cite{css}.

The paper is organized as follows. In Section~\ref{existence} we prove Theorem~\ref{thm:multiplicity}, where we adapt the arguments in \cite{csz}. Theorems~\ref{thm:uncoupled} and~\ref{thm:coupled} are proved in Sections~\ref{limiting} and~\ref{sec:single core} respectively. Section~\ref{segregation} is devoted to the proof of Theorem~\ref{seg:thm}. In Section~\ref{sec:conc} we prove Theorem~\ref{thm:concentr}. A uniform $L^\infty$-estimate is proved in Appendix.

\section{Existence of solutions} \label{existence}

Setting $u_i(x) := \eps^{\frac1{p-1}} v_i(\eps x)$ for $i=1,\ldots,\ell$, the system~\eqref{eq:system} is transformed to
\begin{equation} \label{eq:system2} 
\begin{cases}
-\Delta u_i + \eps^2u_i = \mu_i\what Q_{\eps,i}(x)|u_i|^{2p-2}u_i + \sum\limits_{\substack{j=1 \\ j\neq i}}^\ell\lambda_{ij}|u_j|^p|u_i|^{p-2}u_i, \\
u_i\in H^1(\rn), \quad u_i\ne 0,\quad i=1,\ldots,\ell,
\end{cases}
\end{equation}
where
\begin{equation*}
\what Q_{\eps,i}(x):=
\begin{cases}
1 &\text{if \ }|x-\frac{y_i}{\eps}|<1, \\
-1 &\text{if \ }|x-\frac{y_i}{\eps}|\geq 1.
\end{cases}
\end{equation*} 
Note that $\what Q_{\eps,i}(x) = Q_1(x-\frac{y_i}{\eps})$. Fix $\eps>0$. Set $\cH:=(H^1(\rn))^\ell$ and let the norm of $\bf u=(u_1,\ldots,u_\ell)\in\cH$ be given by 
\[
\|\bf u\|_\eps^2 := \|u_1\|_\eps^2 + \cdots + \|u_\ell\|_\eps^2,
\]
where
\[
\|u_i\|_\eps^2 := \irn(|\nabla u_i|^2 +\eps^2u_i^2), \quad i = 1,\ldots,\ell.
\]
The solutions to~\eqref{eq:system2} are the critical points with nontrivial components of the $\cC^1$-functional $\cJ_\eps:\cH\to\r$ given by
\[
\cJ_\eps(\bf u) := \frac12\|\bf u\|_\eps^2 - \frac1{2p} \sum_{i=1}^\ell\irn \mu_i\what Q_{\eps,i}(x)|u_i|^{2p}  - \frac{1}{2p}\sum_{\substack{i,j=1 \\j\ne i}}^\ell\irn\lambda_{ij}|u_j|^p|u_i|^p.
\]
They belong to the Nehari-type set
\begin{equation*}
\cN_{\eps}:= \Big\{\bf u\in \cH: u_i\neq 0, \ \partial_i\cJ_\eps(\bf u)u_i=0 \ \text{for all} \ i = 1,\ldots,\ell\Big\}
\end{equation*}
introduced in \cite{ctv}. As 
$$\partial_i\cJ_\eps(\bf u)u_i=\|u_i\|_\eps^2 - \irn \mu_i\what Q_{\eps,i}(x)|u_i|^{2p}  - \sum_{\substack{i,j=1 \\j\ne i}}^\ell\irn\lambda_{ij}|u_j|^p|u_i|^p,$$
and $\lambda_{ij}<0$, for every $\bf u\in\cN_\eps$, we have that
\begin{equation}  \label{sobolev}
\|u_i\|_\eps^2 \leq \int_{|x-\frac{y_i}{\eps}|<1} \mu_i|u_i|^{2p} \le C_1\Big(\int_{|x-\frac{y_i}{\eps}|<1}|u_i|^{2^*}\Big)^{2p/2^*} \leq C\Big(\irn |\nabla u_i|^2\Big)^{p}\leq C\|u_i\|_\eps^{2p},
\end{equation}
where the constants $C_1,C>0$ are independent of $\eps$ and $i$.

Hence,
\begin{equation}\label{lbd:sob}
0<C_0\leq \|u_i\|_\eps^2\leq\int_{|x-\frac{y_i}{\eps}|<1} \mu_i|u_i|^{2p}\qquad\text{if \ }\bf u=(u_1,\ldots,u_\ell)\in\cN_\eps,
\end{equation}
where $C_0$ is independent of $\eps$ and $i$. This shows that $\cN_\eps$ is a closed subset of $\cH$. Furthermore, as
\begin{equation}\label{eq:lower bound}
\mathcal{\cJ}_\eps(\bf u) = \frac{p-1}{2p} \|\bf u\|_\eps^2\qquad \text{for all \ }\bf u \in\mathcal{N}_\eps,
\end{equation}
we have that
\begin{equation}\label{eq:infimum}
c_\eps:=\inf_{\cN_\eps}\cJ_\eps\geq\ell C_0=:a_0 \quad \text{for all \ }\eps>0.
\end{equation}
For $\bf t= (t_1,\ldots,t_\ell)\in(0,\infty)^\ell$ and $\bf u\in \cH$ we write  $ \bf t\bf u := (t_1u_1,\ldots,t_\ell u_\ell)$ and define $I_{\eps,\bf u}:(0,\infty)^\ell\to\r$ by
\begin{equation*}
I_{\eps,\bf u}(\bf t) :=\cJ_\eps(\bf t\bf u) = \frac{1}{2}\sum_{i=1}^\ell a_{\eps,\bf u,i} t_i^2 - \frac{1}{2p}\sum_{i=1}^\ell b_{\eps,\bf u,i}t_i^{2p} - \frac{1}{2p}\sum_{\substack{i,j=1 \\ j\ne i}}^\ell d_{\bf u,ij}t_j^pt_i^p,
\end{equation*}
where
$$a_{\eps,\bf u,i}:=\|u_i\|_\eps^2,\qquad b_{\eps,\bf u,i}:=\irn\mu_i\what Q_{\eps,i}(x)|u_i|^{2p},\qquad d_{\bf u,ij}:=\irn\lambda_{ij}|u_j|^p|u_i|^p.$$
As $t_i\partial_i I_{\eps,\bf u}(\bf t)=\partial_i\cJ_\eps(\bf t\bf u)[t_iu_i]$, it follows that, if $u_i\neq 0$ for every $i$, then $\bf t$ is a critical point of $I_{\eps,\bf u}$ if and only if $\bf t\bf u\in\cN_\eps$. If such $\bf t$ exists, it is unique and is a global maximum of $I_{\eps,\bf u}$ (see \cite[Lemma 2.2]{csz} for the proof). Denote this $\bf t$ by $\bf t_{\eps,\bf u}$. Let
\[
S_\eps:=\{v\in H^1(\rn): \|v\|_\eps=1\}, \qquad \mathcal{T}_\eps:=S_\eps\times\cdots\times S_\eps \ (\ell \text{ times}),
\]
\[
\cU_\eps:= \{\mathbf{u}\in \mathcal{T}_\eps: \bf t\mathbf{u}\in \cN_{\eps}\text{ for some (and hence a unique) }\bf t \in(0,\infty)^\ell\}
\]
and let  $\bf m_\eps: \cU_\eps\to \cN_\eps$ be given by $\bf m_\eps(\bf u):=t_{\eps,\bf u}\bf u$.  

\begin{lemma} \label{bfu}\hspace{2em}
\begin{itemize}
\item[$(a)$]$\cU_\eps$ is a nonempty open subset of $\mathcal{T}_\eps$.
\item[$(b)$] $\bf{m}_\eps:\mathcal{U}_\eps\to\mathcal{N}_\eps$ is a homeomorphism and it is odd, i.e., $\bf m_\eps(-\bf u)=-\bf m_\eps(\bf u)$.
\item[$(c)$]If $(\bf u_n)$ is a sequence in $\cU_\eps$ such that $\bf u_n\to \bf u\in\partial\cU_\eps$, then $\|\bf{m}_\eps(\bf u_n)\|_\eps\to\infty$. 
\end{itemize}
\end{lemma}

\begin{proof}
$(a):$ \ Choose $u_i\in H^1(\rn)$ such that $\|u_i\|_\eps^2=1$, $\supp(u_i)\subset B_1(\frac{y_i}{\eps})$ and $\supp(u_i)\cap\supp(u_j)=\emptyset$ if $i\neq j$. Then, $\bf u=(u_1,\ldots,u_\ell)\in\cT_\eps$ and
$$b_{\eps,\bf u,i}+\sum_{j\neq i}d_{\bf u,ij}=b_{\eps,\bf u,i}=\int_{|x-\frac{y_i}{\eps}|<1}|u_i|^{2p}>0.$$
It follows from \cite[Lemma 2.1]{csz} that $\bf u\in\cU_\eps$. As $a_{\eps,\bf u,i}, \ b_{\eps,\bf u,i}, \ d_{\bf u,ij}$ are continuous functions of $\bf u$, \cite[Lemma 2.3]{csz} implies that $\cU_\eps$ is open.

The proofs of the other statements are exactly the same as in \cite[Proposition 3.1]{csz}.
\end{proof}

Clearly $\cT_\eps$ is a $\cC^\infty$-Hilbert submanifold of $\cH$ of codimension $\ell$. Therefore $\cU_\eps$ is a $\cC^\infty$-submanifold of $\cH$. Define $\Psi_\eps:\cU_\eps\to\r$ by
\[
\Psi_\eps(\bf u) := \cJ_\eps(\bf t_{\eps, \bf u}\bf u) = \max_{\bf t\in(0,\infty)^\ell} \cJ_\eps(\bf t\bf u).
\]
If $\Psi_\eps$ is differentiable at $\bf u\in\mathcal{U}_\eps$, the norm of $\Psi_\eps'(\bf u)$ is given by
$$
\|\Psi_\eps'(\bf u)\|_*:=\sup\limits_{\substack{\bf v\in\mathrm{T}_{\bf u}(\mathcal{T_\eps}) \\ \bf v\neq 0}}\frac{|\Psi_\eps'(\bf u)\bf v|}{\|\bf v\|_\eps},
$$
where $\mathrm{T}_{\bf u}(\mathcal{T_\eps})$ is the tangent space to $\mathcal{T_\eps}$ at $\bf u$. A sequence $(\bf u_n)$ in $\mathcal{U_\eps}$ is called a $(PS)_c$\emph{-sequence for} $\Psi_\eps$ if $\Psi_\eps(\bf u_n)\to c$ and $\|\Psi'_\eps(\bf u_n)\|_*\to 0$, and $\Psi_\eps$ is said to satisfy the $(PS)_c$\emph{-condition} if every such sequence has a convergent subsequence.
As usual, a $(PS)_c$\emph{-sequence for} $\mathcal{J}_\eps$ is a sequence $(\bf u_n)$ in $\mathcal{H}$ such that $\mathcal{J}_\eps(\bf u_n)\to c$ and $\|\mathcal{J}_\eps'(\bf u_n)\|_{\mathcal{H}^{-1}}\to 0$, and $\mathcal{J}_\eps$ satisfies the $(PS)_c$\emph{-condition} if any such sequence has a convergent subsequence.

The following result is proved in \cite[Theorem 3.3]{csz}.

\begin{lemma} \label{psi}\hspace{2em}
\begin{itemize}
\item[$(i)$] $\Psi_\eps\in \mathcal{C}^1(\cU_\eps,\r)$ and
\begin{equation*}
\Psi_\eps'(\bf u)\bf v = \cJ_\eps'(\bf{m}_\eps(\bf u))[t_{\eps,\bf u}\bf v] \quad \text{for all } \bf u\in\mathcal{U}_\eps \text{ and } \bf v\in \mathrm{T}_{\bf u}(\mathcal{T}_\eps)
\end{equation*}
where $\mathrm{T}_{\bf u}(\mathcal{T}_\eps)$ is the tangent space to $\cT_\eps$ at $\bf u$.
\item[$(ii)$] If $(\bf u_n)$ is a $(PS)_c$-sequence for $\Psi_\eps$, then $(\bf{m}_\eps(\bf u_n))$ is a $(PS)_c$-sequence for $\cJ_\eps$. Conversely, if $(\bf u_n)$ is a $(PS)_c$-sequence for $\mathcal{J}_\eps$ and $\bf u_n\in\mathcal{N}_\eps$ for all $n\in\mathbb{N}$, then $(\bf{m}_\eps^{-1}(\bf u_n))$ is a $(PS)_c$-sequence for $\Psi_\eps$.
\item[$(iii)$] $\bf u$ is a critical point of $\Psi_\eps$ if and only if $\bf{m}_\eps(\bf u)$ is a critical point of $\cJ_\eps$.
\item[$(iv)$] If $(\bf u_n)$ is a sequence in $\cU_\eps$ such that $\bf u_n\to \bf u\in\partial\cU_\eps$, then $\Psi_\eps(\bf u_n) \to\infty$.
\item[$(v)$] $\Psi_\eps$ is even, i.e., $\Psi_\eps(\bf u)=\Psi_\eps(-\bf u)$ for every $\bf u\in\mathcal{U}_\eps$. 
\end{itemize}
\end{lemma}

As a consequence, we obtain the following result. Its proof is identical to that of \cite[Theorem 3.4]{csz}.

\begin{proposition}\label{prop:existence}\hspace{2em}
\begin{itemize}
\item[$(i)$]If $\Psi_\eps:\cU_\eps\to\r$ satisfies the $(PS)_c$-condition at $c:=\inf_{\cU_\eps}\Psi_\eps$, then the system~\eqref{eq:system2} has a least energy solution $\bf u=(u_1,\ldots,u_\ell)$ which is nonnegative, i.e., $\bf u\in\cN_\eps$, $\cJ_\eps(\bf u)=c_\eps$ and $u_i\ge 0$ for all $i=1,\ldots,\ell$.
\item[$(ii)$]If $\Psi_\eps:\cU_\eps\to\r$ satisfies the $(PS)_c$-condition at every $c\in\r$ and $\gen(\cU_\eps)=\infty$, then the system~\eqref{eq:system2} has an unbounded sequence of solutions in $\cN_\eps$.
\end{itemize}
\end{proposition}

Here ``$\gen$" stands for the Krasnoselskii genus. Recall that if $X$ is a Banach space and $A$ is a subset such that $A=-A$, then $\gen(A)$ is the smallest integer $k$ with the property that there exists an odd map $h: A\to \r^k\smallsetminus\{0\}$. Moreover, $\gen(\emptyset)=0$, and if no $k$ as above exists, then $\gen(A)=\infty$. The properties of genus may be found e.g. in \cite{s}.

Next, we show that the hypotheses of Proposition~\ref{prop:existence} hold true.

\begin{lemma} \label{ps}
$\Psi_\eps:\cU_\eps\to\r$ satisfies the $(PS)_c$-condition at every $c\in\r$.
\end{lemma}

\begin{proof}
Let $(\bf u_n)$ be a $(PS)_c$-sequence for $\mathcal{J}_\eps$ such that $\bf u_n=(u_{n,1},\ldots,u_{n,\ell})\in\mathcal{N}_\eps$ for all $n\in\mathbb{N}$. By Lemma~\ref{psi}, it suffices to show that $(\bf u_n)$ contains a convergent subsequence.

It follows from~\eqref{eq:lower bound} that $(\bf u_n)$ is bounded. Hence,  after passing to a subsequence, $\bf{u}_n \rh \bf u=(u_{1},\ldots,u_{\ell})$ weakly in $\cH$, $\bf{u}_n \to\bf u$ in $L^{2p}_{loc}(\rn)$ and $\bf u_n(x)\to \bf u(x)$ for a.e. $x\in\rn$. Therefore, $\bf u$ is a solution of the system~\eqref{eq:system2} and, using Fatou's lemma, we obtain
\begin{align*}
\|u_i\|_\eps^2 & \le \liminf_{n\to\infty}\|u_{n,i}\|_\eps^2\leq \limsup_{n\to\infty}\|u_{n,i}\|_\eps^2 \\
& =\limsup_{n\to\infty}\Big(\irn\mu_i\what Q_{\eps,i}(x)|u_{n,i}|^{2p}  + \sum_{j\ne i}\irn\lambda_{ij}|u_{n,j}|^p|u_{n,i}|^p\Big) \\
&\leq \lim_{n\to\infty}\int_{|x-\frac{y_i}{\eps}|<1}\mu_i|u_{n,i}|^{2p}-\liminf_{n\to\infty}\int_{|x-\frac{y_i}{\eps}|\geq 1}\mu_i|u_{n,i}|^{2p}-\sum_{j\ne i}\liminf_{n\to\infty}\irn|\lambda_{ij}||u_{n,j}|^p|u_{n,i}|^p \\
&\leq \int_{|x-\frac{y_i}{\eps}|<1}\mu_i|u_i|^{2p} - \int_{|x-\frac{y_i}{\eps}|\geq 1}\mu_i|u_i|^{2p}-\sum_{j\ne i}\irn|\lambda_{ij}||u_j|^p|u_i|^p \\
&=\irn \mu_i\what Q_{\eps,i}(x)|u_i|^{2p}  + \sum_{j\ne i}\irn\lambda_{ij}|u_j|^p|u_i|^p = \|u_i\|_\eps^2.
\end{align*}
This shows that $\lim_{n\to\infty}\|u_{n,i}\|_\eps = \|u_i\|_\eps$ and, as a consequence, $\bf u_n\to \bf u$ strongly in $\cH$. 
\end{proof}

\begin{lemma}\label{lem:genus}
$\gen(\cU_\eps)=\infty$.
\end{lemma}

\begin{proof}
Fix $k\geq 1$. For each $j=1,\ldots,k$ and $i=1,\ldots,\ell$ we choose $u_{j,i}\in H^1(\rn)$ such that $\|u_{j,i}\|_\eps=1$, $\supp(u_{j,i})\subset B_1(\frac{y_i}{\eps})$ and $\mathrm{supp}(u_{j,i})\cap \mathrm{supp}(u_{j',i'})=\emptyset$ if $(j,i)\neq (j',i')$.
Let $\{e_j : 1 \leq j \leq k\}$ be the standard basis of $\r^k$, and set
$$\mathbb{P}^{k-1}:= \Big\{\sum_{j=1}^k r_j \hat{e}_j : \hat{e}_j \in \{\pm e_j\}, \, r_j \in [0,1], \, \sum_{j=1}^k r_j = 1 \Big\}.$$
$\mathbb{P}^{k-1}$ is homeomorphic to the unit sphere $\mathbb{S}^{k-1}$ in $\r^k$ by an odd homeomorphism.
For each $i=1,\ldots,\ell$, let $\sigma_i:\mathbb{P}^{k-1}\to H^1(\rn)$ be given by $\sigma_i(e_j) := u_{j,i}$, \ $\sigma_i(-e_j):= -u_{j,i}$, and
$$\sigma_i\Big(\sum_{j=1}^k r_j \hat{e}_j \Big) := \frac{\sum_{j=1}^k r_j \sigma_i(\hat{e}_j)}{\|\sum_{j=1}^k r_j \sigma_i(\hat{e}_j)\|_\eps}.$$
Since $u_{j,i}$ and $u_{j',i'}$ have disjoint supports if $(j,i)\neq (j',i')$, these maps are well defined and $\mathrm{supp}(\sigma_i(z))\cap\mathrm{supp}(\sigma_{i'}(z))=\emptyset$ if $i\neq i'$ for every $z\in \mathbb{P}^{k-1}$. Arguing as in Lemma~\ref{bfu}$(a)$, we see that the map $\sigma:\mathbb{P}^{k-1}\to\cU_\eps$ given by $\sigma(z):=(\sigma_1(z),\ldots,\sigma_\ell(z))$ is well defined. As each $\sigma_i$ is continuous and odd, so is $\sigma$. By the monotonicity property of the genus we have that $\gen(\cU_\eps) \geq \gen(\mathbb{P}^{k-1})=k$. Since $k$ is arbitrary, the conclusion follows.
\end{proof}
\medskip

\begin{proof}[Proof of Theorem~\ref{thm:multiplicity}]
This follows immediately from Proposition~\ref{prop:existence} and Lemmas~\ref{ps} and~\ref{lem:genus}.
\end{proof}

\section{The limit profile of minimizers for different regions of attraction} \label{limiting}

Throughout this section we assume that $y_i\neq y_j$ if $i\neq j$. Recall that  $E:=D^{1,2}(\rn)\cap L^{2p}(\rn)$ is the Banach space whose norm is given by
$$\|u\|_E^2:=\|u\|^2+|u|_{2p}^2,\qquad\text{where \ } \ \|u\|^2:=\irn|\nabla u|^2\text{ \ and \ }|u|_{2p}^{2p}:=\irn|u|^{2p}.$$

The main objective of this section is to show the following.

\begin{theorem} \label{convergence}
Let $\eps_n\to 0$ and $\bf u_n=(u_{n,1},\ldots,u_{n,\ell})$ be a nonnegative least energy solution of the system~\eqref{eq:system2} with $\eps=\eps_n$. Then, after passing to a subsequence, $u_{n,i}(\,\cdot\,+\frac{y_i}{\eps_n})\to w_i$ strongly in $E$, where $w_i$ is a positive least energy solution to the equation~\eqref{eq:limiting}.
\end{theorem}

The solutions of~\eqref{eq:limiting} are the nontrivial critical points of the functional $J_i:E\to\r$ given by
$$J_i(w):=\frac{1}{2}\|w\|^2-\frac{1}{2p}\irn\mu_iQ_1(x)|w|^{2p}.$$
They belong to the Nehari manifold
$$\cM_i:=\{w\in E:w\neq 0, \ J_i'(w)w=0\}.$$
Set 
$$\kappa_i:=\inf_{\cM_i}J_i.$$
A solution $w\in\cM_i$ that satisfies $J_i(w)=\kappa_i$ is called a \emph{least energy solution}. The existence of such solution is proved in \cite{fw}.

We start with the following lemma.

\begin{lemma} \label{unifbound}
There exist $\eps_0>0$ and constants $a_0,d_0>0$ such that $a_0 \le c_\eps\le d_0$ for all $\eps\in(0,\eps_0)$.
\end{lemma}

\begin{proof}
The lower bound $a_0$ is given in~\eqref{eq:infimum}.

Let $\eps_0:=\frac{1}{2}\min_{i\neq j}|y_i-y_j|$ and fix $\vp\in\cC_c^\infty(B_1(0))$, $\vp\neq 0$. For each $\eps\in(0,\eps_0]$, let 
\begin{equation} \label{tei}
t_{\eps,i}:=\left(\frac{\|\vp\|_\eps^2}{\int_{B_1(0)}\mu_i|\vp|^{2p}}\right)^{1/(2p-2)}=\left(\frac{\int_{B_1(0)}(|\nabla \vp|^2+\eps^2\vp^2)}{\int_{B_1(0)}\mu_i|\vp|^{2p}}\right)^{1/(2p-2)}.
\end{equation} 
Set $v_{\eps,i}(x):=t_{\eps,i}\vp(x-\frac{y_i}{\eps})$ and $\bf v_\eps:=(v_{\eps,1},\ldots,v_{\eps,\ell})$. Then 
$$\|v_{\eps,i}\|_\eps^2=\int_{B_1(\frac{y_i}{\eps})}\mu_i|v_{\eps,i}|^{2p}=\irn\mu_i\what Q_{\eps,i}|v_{\eps,i}|^{2p}$$
and, as $v_{\eps,i}$ and $v_{\eps,j}$ have disjoint supports if $i\neq j$, we have that $\bf v_\eps\in\cN_\eps$. It follows that
\begin{align*}
\frac{2p}{p-1}c_\eps &\leq\|\bf v_\eps\|_\eps^2=\sum_{i=1}^\ell \int_{B_1(\frac{y_i}{\eps})}(|\nabla v_{\eps,i}|^2+\eps^2v_{\eps,i}^2)=\sum_{i=1}^\ell t_{\eps,i}^2\int_{B_1(0)}(|\nabla \vp|^2+\eps^2\vp^2) \\
&\leq C\sum_{i=1}^\ell\int_{B_1(0)}(|\nabla \vp|^2+\vp^2)=:\frac{2p}{p-1} d_0\qquad\text{for every \ }\eps\in(0,\eps_0],
\end{align*}
as claimed.
\end{proof}
\medskip

\begin{proof}[Proof of Theorem~\ref{convergence}]
Let $\bf u_n=(u_{n,1},\ldots,u_{n,\ell})\in\cN_{\eps_n}$ satisfy $\cJ_{\eps_n}(\bf u_n)=c_{\eps_n}$ and $u_{n,i}\geq 0$. By Lemma~\ref{unifbound}, $(u_{n,i})$ is bounded in $D^{1,2}(\rn)$ and, as $\lambda_{ij}<0$, using~\eqref{sobolev} and  Lemma~\ref{unifbound}, we obtain
\begin{align*}
\int_{|x-\frac{y_j}{\eps_n}|\geq 1}\mu_i|u_{n,i}|^{2p}&\leq \|u_{n,i}\|_{\eps_n}^2+\int_{|x-\frac{y_i}{\eps_n}|\geq 1}\mu_i|u_{n,i}|^{2p}-\sum\limits_{j\neq i}\irn\lambda_{ij}|u_{n,j}|^p|u_{n,i}|^p \\
&=\int_{|x-\frac{y_i}{\eps_n}|< 1}\mu_i|u_{n,i}|^{2p}\leq C\|u_{n,i}\|_{\eps_n}^{2p}\leq C_1.
\end{align*}
This shows that $(u_{n,i})$ is bounded in $L^{2p}(\rn)$ and, therefore, in $E$. Set
$$w_{n,i,j}(x):=u_{n,i}\Big(x+\frac{y_j}{\eps_n}\Big).$$
Then, $(w_{n,i,j})$ is bounded in $E$.  So there exists $w_{i,j}\in E$ such that, after passing to a subsequence, $w_{n,i,j}\rh w_{i,j}$ weakly in $E$, $w_{n,i,j}\to w_{i,j}$ in $L^{2p}_{loc}(\rn)$ and $w_{n,i,j}\to w_{i,j}$ a.e. in $\rn$. Hence $w_{i,j}\geq 0$. From~\eqref{eq:lower bound} and~\eqref{eq:infimum} we get
$$0<\frac{2p}{p-1}a_0\leq\|u_{n,i}\|_{\eps_n}^2\leq\int_{|x-\frac{y_i}{\eps_n}|< 1}\mu_i|u_{n,i}|^{2p}=\int_{|x|< 1}\mu_i|w_{n,i,i}|^{2p},$$
and passing to the limit we see that $\int_{|x|< 1}|w_{i,i}|^{2p}>0$. Hence, $w_{i,i}\neq 0$. Let $\vp\in\cC^\infty_c(\rn)$ and set $\vp_{n,j}(x):=\vp(x-\frac{y_j}{\eps_n})$. Then, performing the change of variables $x\mapsto x+\frac{y_j}{\eps_n}$ and recalling that $\what Q_{\eps_n,i}(x) = Q_1(x-\frac{y_i}{\eps_n})$, we obtain
\begin{align}
0&=\partial_i\cJ_{\eps_n}(\bf u_n)\vp_{n,j} \nonumber \\
&=\irn\Big(\nabla u_{n,i}\cdot\nabla\vp_{n,j} + \eps_n^2 u_{n,i}\vp_{n,j}-\mu_i\what Q_{\eps_n,i}(x)|u_{n,i}|^{2p-2}u_{n,i}\vp_{n,j}-\sum_{k\neq i}\lambda_{ik}|u_{n,k}|^p|u_{n,i}|^{p-2}u_{n,i}\vp_{n,j}\Big)\nonumber \\
&=\irn\Big(\nabla w_{n,i,j}\cdot\nabla\vp + \eps_n^2 w_{n,i,j}\vp-\mu_i Q_1\big(x-\frac{y_i-y_j}{\eps_n}\big)|w_{n,i,j}|^{2p-2}w_{n,i,j}\vp\nonumber \\
&\qquad\qquad-\sum_{k\neq i}\lambda_{ik}|w_{n,k,j}|^p|w_{n,i,j}|^{p-2}w_{n,i,j}\vp\Big). \label{eq:partial J}
\end{align}
If $j\neq i$, then $\frac{|y_i-y_j|}{\eps_n}\to\infty$. Since $\vp$ has compact support, letting $n\to\infty$ we obtain
\begin{align*}
\irn\nabla w_{i,j}\cdot\nabla\vp=-\irn\mu_i|w_{i,j}|^{2p-2}w_{i,j}\vp+\sum_{k\neq i}\irn\lambda_{ik}|w_{k,j}|^p|w_{i,j}|^{p-2}w_{i,j}\vp\qquad\text{for every \ }\vp\in\cC^\infty_c(\rn).
\end{align*}
As $w_{i,j}\geq 0$, this implies that
$$\irn\nabla w_{i,j}\cdot\nabla\vp\leq 0\qquad\text{for every \ }\vp\in\cC^\infty_c(\rn)\text{ \ with \ }\vp\geq 0,$$
and, as a consequence, $w_{i,j}=0$, i.e., $w_{n,i,j}\rh 0$ weakly in $E$ if $i\neq j$. Therefore, setting $i=j$ in~\eqref{eq:partial J} and letting $n\to\infty$ we get that
\[
\irn\nabla w_{i,i}\cdot\nabla\vp=\irn\mu_iQ_1(x)|w_{i,i}|^{2p-2}w_{i,i}\vp\qquad\text{for every \ }\vp\in\cC^\infty_c(\rn).
\]
This shows that $w_i:=w_{i,i}$ solves~\eqref{eq:limiting}. From this fact and Fatou's lemma we obtain
\begin{align}\label{eq:w_i}
\|w_i\|^2 & \leq \liminf_{n\to\infty}\|w_{n,i,i}\|^2\leq \limsup_{n\to\infty}\|w_{n,i,i}\|^2= \limsup_{n\to\infty}\|u_{n,i}\|^2\leq \limsup_{n\to\infty}\|u_{n,i}\|_{\eps_n}^2 \nonumber \\
&=\limsup_{n\to\infty}\Big(\irn\mu_i\what Q_{\eps_n,i}(x)|u_{n,i}|^{2p} + \sum_{j\neq i}\irn\lambda_{ij}|u_{n,j}|^p|u_{n,i}|^p\Big) \nonumber \\
&=\limsup_{n\to\infty}\Big(\irn\mu_i Q_1(x)|w_{n,i,i}|^{2p} + \sum_{j\neq i}\irn\lambda_{ij}|w_{n,j,i}|^p|w_{n,i,i}|^p\Big) \nonumber \\
&\leq\limsup_{n\to\infty}\irn\mu_i Q_1(x)|w_{n,i,i}|^{2p} \leq \irn\mu_i Q_1(x)|w_i|^{2p} = \|w_i\|^2.
\end{align} 
It follows that $w_{n,i,i}\to w_i$ strongly in $D^{1,2}(\rn)$. Replacing $\limsup$ by $\liminf$ in the first three lines of \eqref{eq:w_i} we obtain
\begin{align*}
\|w_i\|^2 & \leq \liminf_{n\to\infty}\Big(\irn\mu_i Q_1(x)|w_{n,i,i}|^{2p} + \sum_{j\neq i}\irn\lambda_{ij}|w_{n,j,i}|^p|w_{n,i,i}|^p\Big)\leq \liminf_{n\to\infty}\irn\mu_i Q_1(x)|w_{n,i,i}|^{2p} \nonumber \\
&\leq\limsup_{n\to\infty}\irn\mu_i Q_1(x)|w_{n,i,i}|^{2p} \leq \irn\mu_i Q_1(x)|w_i|^{2p} = \|w_i\|^2.
\end{align*} 
Hence,
$$\lim_{n\to\infty}\irn\mu_i Q_1(x)|w_{n,i,i}|^{2p}=\irn\mu_i Q_1(x)|w_i|^{2p}.$$
As $w_{n,i,i}\to w_i$ in $L_{loc}^{2p}(\rn)$ it follows that $w_{n,i,i}\to w_i$ strongly in $L^{2p}(\rn)$, thus, in $E$.
Furthermore,
\begin{equation} \label{eq:energy}
\lim_{n\to\infty}\cJ_{\eps_n}(\bf u_n)=\frac{p-1}{2p}\sum_{i=1}^\ell \lim_{n\to\infty}\|u_{n,i}\|_{\eps_n}^2=\frac{p-1}{2p}\sum_{i=1}^\ell \|w_i\|^2.
\end{equation}
Note also for further reference that 
\begin{equation} \label{eq:reference}
\irn \eps_n^2 u_{n,i}^2\to 0.
\end{equation}

To prove that all $w_i$ are least energy solutions of~\eqref{eq:limiting} we argue by contradiction. Assume that (at least) one of them is not. We can choose $v_i\in\cM_i\cap\cC^\infty_c(\rn)$, $1\le i\le \ell$, such that
\[
\sum_{i=1}^\ell \|w_i\|^2 > \sum_{i=1}^\ell \|v_i\|^2.
\]
Set
$$ 
v_{n,i}(x) := v_i\Big(x-\frac{y_i}{\eps_n}\Big)\text{ \ for \ }i=1,\ldots,\ell\qquad\text{and}\qquad\bf v_n:=(v_{n,1},\ldots,v_{n,\ell}).
$$
Then $\bf v_n\in\cH$. Since $|\frac{y_i-y_j}{\eps_n}|\to\infty$ if $i\neq j$, we have that $\lim_{n\to\infty}\irn\lambda_{ij}|v_{n,j}|^p|v_{n,i}|^p=0$. Therefore,
\begin{equation} \label{pos}
\irn\mu_i\what Q_{\eps_n,i}|v_{n,i}|^{2p}+\sum_{j\neq i}\irn\lambda_{ij}|v_{n,j}|^p|v_{n,i}|^p > 0
\end{equation}
for $n$ large enough. Hence, there exist $\bf t_n=(t_{n,1},\ldots,t_{n,\ell})\in(0,\infty)^\ell$ such that $\bf t_n\bf v_n\in\cN_{\eps_n}$; see \cite[Lemma 2.1]{csz}. Moreover, $t_{n,i}$ are bounded and bounded away from 0. Indeed,
\begin{align*}
0 = \partial_{t_i}\cJ_{\eps_n}(\bt_n\bf v_n) & = t_{n,i}\|v_{n,i}\|_{\eps_n}^2 - t_{n,i}^{2p-1}\irn\mu_i\what Q_{\eps_n,i}|v_{n,i}|^{2p} - t_{n,j}^pt_{n,i}^{p-1}\sum_{j\neq i}\irn\lambda_{ij}|v_{n,j}|^p|v_{n,i}|^p \\
& \ge t_{n,i}\|v_{n,i}\|_{\eps_n}^2 - t_{n,i}^{2p-1}\irn\mu_i\what Q_{\eps_n,i}|v_{n,i}|^{2p},
\end{align*}
hence $t_{n,i}$ is bounded away from 0. For each $n$ choose $i(n)$ such that $t_{n,i(n)}\ge t_{n,j}$ for all $j$. It follows then from the first line above and~\eqref{pos} that $t_{n,i(n)}$ is bounded, and hence so are all $t_{n,i}$. Passing to a subsequence, $t_{n,i}\to t_i$ where $t_i$ is bounded and bounded away from 0. We have
\begin{align*}
t_{n,i}^{2p}\irn\mu_iQ_1|v_i|^{2p} + o(1) & = t_{n,i}^{2p}\irn\mu_i\what Q_{\eps_n,i}|v_{n,i}|^{2p} + o(1) \\
& = t_{n,i}^{2p}\irn\mu_i\what Q_{\eps_n,i}|v_{n,i}|^{2p} + \sum_{j\neq i}t_{n,j}^pt_{n,i}^p\irn\lambda_{ij}|v_{n,j}|^p|v_{n,i}|^p
 = t_{n,i}^2\|v_{n,i}\|_{\eps_n}^2,
\end{align*}
and passing to the limit,
\[
t_i^{2p}\irn\mu_iQ_1|v_i|^{2p} = t_i^2\|v_i\|^2.
\]
As $v_i\in \cM_i$,  $t_i=1$. Since $\bf u_n$ is a least energy solution to~\eqref{eq:system2},
\begin{align*}
\lim_{n\to\infty}\cJ_{\eps_n}(\bf u_n)&\leq\lim_{n\to\infty}\cJ_{\eps_n}(\bf t_n\bf v_n)=\frac{p-1}{2p}\sum_{i=1}^\ell \lim_{n\to\infty}\|t_{n,i}v_{n,i}\|_{\eps_n}^2 \\
&=\frac{p-1}{2p}\sum_{i=1}^\ell \|v_i\|^2 <\frac{p-1}{2p}\sum_{i=1}^\ell\|w_i\|^2.
\end{align*}
This contradicts~\eqref{eq:energy}. Hence, $w_i$ are nonnegative least energy solutions of~\eqref{eq:limiting} and the (strict) positivity of $w_i$ follows by the maximum principle (see \cite[Theorem 2.5]{chs}).
\end{proof}

\begin{corollary} \label{concentration}
Let $\eps_n\to 0$ and let $\bf v_n=(v_{n,1},\ldots,v_{n,\ell})$ be a nonnegative least energy solution to the system~\eqref{eq:system} with $\eps=\eps_n$. For each $\delta>0$, after passing to a subsequence, the following statements hold true:
\[
\lim_{n\to\infty} \frac{\int_{|x-y_i|\leq\delta}(|\nabla v_{n,i}|^2+v_{n,i}^2)}{\irn (|\nabla v_{n,i}|^2+v_{n,i}^2)} = 1 \qquad \text{and} \qquad \lim_{n\to\infty} \frac{\int_{|x-y_i|\leq\delta}|v_{n,i}|^{2p}}{\irn |v_{n,i}|^{2p}} = 1.
\]
\end{corollary}

\begin{proof}
Let $\bf u_n=(u_{n,1},\ldots,u_{n,\ell})$ be given by $u_{n,i}(x) = \eps_n^{\frac1{p-1}}v_{n,i}(\eps_n x)$. Then $\bf u_n$ is a least energy solution to the system~\eqref{eq:system2}. Setting $w_{n,i}(x) := u_{n,i}(x+\frac{y_i}{\eps_n})=\eps_n^\frac{1}{p-1}v_{n,i}(\eps_nx+y_i)$ and performing a change of variable we obtain
\[
 \frac{\int_{|x-y_i|\leq\delta}(|\nabla v_{n,i}|^2+v_{n,i}^2)}{\irn (|\nabla v_{n,i}|^2+v_{n,i}^2)} = \frac{\eps_n^{N-\frac{2p}{p-1}}\int_{|x|\leq\frac{\delta}{\eps_n}}(|\nabla w_{n,i}|^2+\eps_n^2w_{n,i}^2)}{\eps_n^{N-\frac{2p}{p-1}}\irn (|\nabla w_{n,i}|^2+\eps_n^2w_{n,i}^2)}  = \frac{\int_{|x|\leq\frac{\delta}{\eps_n}}(|\nabla w_{n,i}|^2+\eps_n^2w_{n,i}^2)}{\irn (|\nabla w_{n,i}|^2+\eps_n^2w_{n,i}^2)}.
\]
According to Theorem~\ref{convergence}, $w_{n,i}\to w_i$ strongly in $E$ and, by~\eqref{eq:reference}, $\eps_n w_{n,i}\to 0$ in $L^2(\rn)$. It follows that the numerator and the denominator on the right-hand side above tend to $\|w_i\|^2$. Hence the first statement. The proof of the second one is somewhat simpler:
\[
 \frac{\int_{|x-y_i|\leq\delta}|v_{n,i}|^{2p}}{\irn |v_{n,i}|^{2p}} = \frac{\eps_n^{N-\frac{2p}{p-1}}\int_{|x|\leq\frac{\delta}{\eps_n}}|w_{n,i}|^{2p}}{\eps_n^{N-\frac{2p}{p-1}}\irn |w_{n,i}|^{2p}}  = \frac{\int_{|x|\leq\frac{\delta}{\eps_n}}|w_{n,i}|^{2p}}{\irn |w_{n,i}|^{2p}} \to 1.
\]
This completes the proof.
\end{proof}

\medskip

\begin{proof}[Proof of Theorem~\ref{thm:uncoupled}]
This follows immediately from Theorem~\ref{convergence} and Corollary~\ref{concentration}.
\end{proof}

\section{The limit profile of minimizers for a single region of attraction}
\label{sec:single core}

Throughout this section we assume that $y_1=\cdots=y_\ell=0$.  Note that here we have $\what Q_{\eps,i}(x) = Q_1(x)$ for all $i$ and hence 
\[
\cJ_\eps(\bf u) = \frac12\|\bf u\|_\eps^2 - \frac1{2p} \sum_{i=1}^\ell\irn \mu_i Q_1(x)|u_i|^{2p}  - \frac{1}{2p}\sum_{\substack{i,j=1 \\j\ne i}}^\ell\irn\lambda_{ij}|u_j|^p|u_i|^p.
\] 

\begin{theorem} \label{thm:limit for single core}
Let $\eps_n\to 0$ and let $\bf u_n$ be a nonnegative least energy solution of the system~\eqref{eq:system2} with $\eps=\eps_n$. Then, after passing to a subsequence, $(\bf u_n)$ converges strongly in $E^\ell$ to a nonnegative least energy solution of the limit system~\eqref{eq:system3}.
\end{theorem}
\medskip

The solutions to~\eqref{eq:system3} are the critical points with nontrivial components of the functional $\cJ_0:E^\ell\to\r$ given by
$$\cJ_0(\bf u) := \frac12\|\bf u\|^2 - \frac1{2p} \sum_{i=1}^\ell\irn \mu_iQ_1(x)|u_i|^{2p}  - \frac{1}{2p}\sum_{\substack{i,j=1 \\j\ne i}}^\ell\irn\lambda_{ij}|u_j|^p|u_i|^p,$$
where as before, 
\begin{equation} \label{norm}
\|\bf u\|^2:=\|u_1\|^2+\cdots+\|u_\ell\|^2\qquad\text{and}\qquad\|u_i\|^2:=\irn |\nabla u_i|^2.
\end{equation}
They belong to the Nehari-type set
\begin{equation*}
\cN_0:= \Big\{\bf u\in E^\ell: u_i\neq 0, \ \partial_i\cJ_0(\bf u)u_i=0 \ \text{for all} \ i = 1,\ldots,\ell\Big\}.
\end{equation*}
Arguing as in~\eqref{sobolev}-\eqref{lbd:sob} we see that $0<\what C_0\leq\|u_i\|^2\leq\|u_i\|^2_E$ for every $\bf u=(u_1,\ldots,u_\ell)\in\cN_0$. Therefore, $\cN_0$ is closed in $E^\ell$ and
$$
c_0:=\inf_{\cN_0}\cJ_0\geq\what a_0>0
$$
(cf.~\eqref{eq:infimum}).
By a least energy solution of~\eqref{eq:system3} we mean a solution $\bf u$ that satisfies $\cJ_0(\bf u)=c_0$.

As in Section~\ref{limiting}, we have the following lemma.

\begin{lemma} \label{lem:bounds single core}
There exist constants $a_0,d_1>0$ such that $a_0\leq c_\eps\leq d_1$ for all $\eps\in(0,1]$.
\end{lemma}

\begin{proof}
The lower bound $a_0$ is given in~\eqref{eq:infimum}.

Let $\vp_1,\ldots,\vp_\ell\in \cC_c^\infty(B_1(0))\smallsetminus\{0\}$ be such that $\vp_i$ and $\vp_j$ have disjoint supports if $i\ne j$, and let $t_{\eps,i}$ be as in~\eqref{tei}, with $\vp$ replaced by $\vp_i$. Set $v_{\eps,i} := t_{\eps,i}\vp_i$ and $\bf v_\eps := (v_{\eps,1},\ldots,v_{\eps,\ell})$. The rest of the argument is exactly the same as in the proof of Lemma~\ref{unifbound}.
\end{proof}
\medskip

\begin{proof}[Proof of Theorem~\ref{thm:limit for single core}]
Let $\bf u_n=(u_{n,1},\ldots,u_{n,\ell})\in\cN_{\eps_n}$ satisfy $\cJ_{\eps_n}(\bf u_n)=c_{\eps_n}$ and $u_{n,i}\geq 0$. By Lemma~\ref{lem:bounds single core}, $(u_{n,i})$ is bounded in $D^{1,2}(\rn)$ and, thus, in $L^{2p}_{loc}(\rn)$. As
\begin{equation}\label{2pbd}
\int_{|x|\geq 1}\mu_i|u_{n,i}|^{2p}\leq \|u_{n,i}\|_{\eps_n}^2+\int_{|x|\geq 1}\mu_i|u_{n,i}|^{2p}-\sum\limits_{j\neq i}\irn\lambda_{ij}|u_{n,j}|^p|u_{n,i}|^p =\int_{|x|< 1}\mu_i|u_{n,i}|^{2p},
\end{equation}
we have that $(u_{n,i})$ is bounded in $L^{2p}(\rn)$ and, therefore, in $E$. So, after passing to a subsequence, $u_{n,i}\rh u_i$ weakly in $E$, $u_{n,i}\to u_i$ in $L^{2p}_{loc}(\rn)$ and $u_{n,i}\to u_i$ a.e. in $\rn$. Hence, $u_i\geq 0$. From~\eqref{eq:lower bound} and~\eqref{eq:infimum} we get that
$$0<\frac{2p}{p-1}a_0\leq\|u_{n,i}\|_{\eps_n}^2\leq\int_{|x|< 1}\mu_i|u_{n,i}|^{2p},$$
and passing to the limit we see that $\int_{|x|< 1}|u_i|^{2p}>0$. Hence, $u_i\neq 0$. Let $\vp\in\cC^\infty_c(\rn)$. Then,
\begin{align*}
0=\partial_i\cJ_{\eps_n}(\bf u_n)\vp=\irn\Big(\nabla u_{n,i}\cdot\nabla\vp + \eps_n^2 u_{n,i}\vp-\mu_iQ_1(x)|u_{n,i}|^{2p-2}u_{n,i}\vp-\sum_{j\neq i}\lambda_{ij}|u_{n,j}|^p|u_{n,i}|^{p-2}u_{n,i}\vp\Big),
\end{align*}
and letting $n\to\infty$ we obtain
\begin{align*}
\irn\nabla u_i\cdot\nabla\vp=\irn\mu_iQ_1(x)|u_i|^{2p-2}u_i\vp+\sum_{j\neq i}\irn\lambda_{ij}|u_j|^p|u_i|^{p-2}u_i\vp\qquad\text{for every \ }\vp\in\cC^\infty_c(\rn).
\end{align*}
This shows that $\bf u=(u_1,\ldots,u_\ell)$ solves~\eqref{eq:system3}. Using Fatou's lemma and the fact that $\lambda_{ij}<0$ and $\bf u\in\cN_0$ we get
\begin{align*}
\|u_i\|^2 &\leq\liminf_{n\to\infty}\|u_{n,i}\|^2\leq\limsup_{n\to\infty}\|u_{n,i}\|^2\leq\limsup_{n\to\infty}\|u_{n,i}\|_{\eps_n}^2\\
&=\limsup_{n\to\infty}\Big(\irn\mu_iQ_1(x)|u_{n,i}|^{2p}+\sum_{j\neq i}\irn\lambda_{ij}|u_{n,j}|^p|u_{n,i}|^p\Big)\\
&\leq\lim_{n\to\infty}\int_{|x|< 1}\mu_i|u_{n,i}|^{2p}-\liminf_{n\to\infty}\int_{|x|\geq 1}\mu_i|u_{n,i}|^{2p}-\sum_{j\neq i}\liminf_{n\to\infty}\irn|\lambda_{ij}||u_{n,j}|^p|u_{n,i}|^p \\
&\leq \int_{|x|< 1}\mu_i|u_{i}|^{2p}-\int_{|x|\geq 1}\mu_i|u_i|^{2p}-\sum_{j\neq i}\irn|\lambda_{ij}||u_j|^p|u_i|^p \\
&=\irn\mu_iQ_1(x)|u_i|^{2p}+\sum_{j\neq i}\irn\lambda_{ij}|u_j|^p|u_i|^p=\|u_i\|^2\qquad\text{for every \ }i=1,\ldots,\ell.
\end{align*}
This shows that $u_{n,i}\to u_i$ strongly in $D^{1,2}(\rn)$. Replacing $\limsup$ with $\liminf$ in the first two lines of the display above and recalling that $u_{n,i}\to u_i$ in $L^{2p}_{loc}(\rn)$ we obtain
$$\lim_{n\to\infty}\Big(\int_{|x|\geq 1}\mu_i|u_{n,i}|^{2p}+\sum_{j\neq i}\irn|\lambda_{ij}||u_{n,j}|^p|u_{n,i}|^p\Big)=\int_{|x|\geq 1}\mu_i|u_i|^{2p}+\sum_{j\neq i}\irn|\lambda_{ij}||u_j|^p|u_i|^p.$$
As
\begin{align*}
&\int_{|x|\geq 1}\mu_i|u_i|^{2p}\leq\liminf_{n\to\infty}\int_{|x|\geq 1}\mu_i|u_{n,i}|^{2p}\quad\text{and}\quad\sum_{j\neq i}\irn|\lambda_{ij}||u_j|^p|u_i|^p\leq\liminf_{n\to\infty}\sum_{j\neq i}\irn|\lambda_{ij}||u_{n,j}|^p|u_{n,i}|^p
\end{align*}
from the second inequality we derive
\begin{align*}
\int_{|x|\geq 1}\mu_i|u_i|^{2p}&=\Big(\int_{|x|\geq 1}\mu_i|u_i|^{2p}+\sum_{j\neq i}\irn|\lambda_{ij}||u_j|^p|u_i|^p\Big) - \sum_{j\neq i}\irn|\lambda_{ij}||u_j|^p|u_i|^p \\
&\geq \lim_{n\to\infty}\Big(\int_{|x|\geq 1}\mu_i|u_{n,i}|^{2p}+\sum_{j\neq i}\irn|\lambda_{ij}||u_{n,j}|^p|u_{n,i}|^p\Big)+\limsup_{n\to\infty}\Big(-\sum_{j\neq i}\irn|\lambda_{ij}||u_{n,j}|^p|u_{n,i}|^p\Big) \\
&\geq \limsup_{n\to\infty}\Big(\int_{|x|\geq 1}\mu_i|u_{n,i}|^{2p}+\sum_{j\neq i}\irn|\lambda_{ij}||u_{n,j}|^p|u_{n,i}|^p-\sum_{j\neq i}\irn|\lambda_{ij}||u_{n,j}|^p|u_{n,i}|^p\Big)\\ &=\limsup_{n\to\infty}\int_{|x|\geq 1}\mu_i|u_{n,i}|^{2p}.
\end{align*}
Therefore,
$$\lim_{n\to\infty}\int_{|x|\geq 1}\mu_i|u_{n,i}|^{2p}=\int_{|x|\geq 1}\mu_i|u_i|^{2p}.$$
As a consequence, $u_{n,i}\to u_i$ strongly in $L^{2p}(\rn)$ and, thus, in $E$. It follows that
\begin{equation}\label{eq:limJ}
\lim_{n\to\infty}\cJ_{\eps_n}(\bf u_n)=\cJ_0(\bf u).
\end{equation}

To show that $\bf u$ is a least energy solution we argue by contradiction. Assume that $\cJ_0(\bf u)>c_0$. Let $\bf v\in\cN_0\cap\cC^\infty_c(\rn)^\ell$ be such that $\cJ_0(\bf u)>\cJ_0(\bf v)\geq c_0$. As $\bf v\in\cH$ and
\begin{align*}
0 <\|v_i\|^2=\irn\mu_iQ_1(x)|v_i|^{2p} + \sum\limits_{j\neq i}\irn\lambda_{ij}|v_j|^p|v_i|^p,
\end{align*}
for each $\eps_n$ there exists $\bf t_{n}=(t_{n,1},\ldots,t_{n,\ell})\in(0,\infty)^\ell$ such that $\bf t_{n}\bf v\in\cN_{\eps_n}$. As in the proof of Theorem~\ref{convergence} we see that $t_{n,i}\to 1$ for all $i$. Therefore, using~\eqref{eq:limJ} we get
$$
\cJ_0(\bf u) = \lim_{n\to\infty} \cJ_{\eps_n}(\bf u_n) \le \lim_{n\to\infty} \cJ_{\eps_n}(\bf t_n\bf v) = \lim_{n\to\infty}\Big( \cJ_0(\bf t_n\bf v)+\sum_{i=1}^\ell t_{n,i}^2\irn\eps_n^2v_i^2\Big) = \cJ_0(\bf v) < \cJ_0(\bf u).
$$
This is a contradiction.
\end{proof}

\begin{corollary} \label{concentration2}
Let $\eps_n\to 0$ and let $\bf v_n=(v_{n,1},\ldots,v_{n,\ell})$ be a nonnegative least energy solution to the system~\eqref{eq:system} with $\eps=\eps_n$. For each $\delta>0$, after passing to a subsequence, the following statements hold true:
\[
\lim_{n\to\infty} \frac{\int_{|x|\leq\delta}(|\nabla v_{n,i}|^2+v_{n,i}^2)}{\irn (|\nabla v_{n,i}|^2+v_{n,i}^2)} = 1 \qquad \text{and} \qquad \lim_{n\to\infty} \frac{\int_{|x|\leq\delta}|v_{n,i}|^{2p}}{\irn |v_{n,i}|^{2p}} = 1.
\]
\end{corollary}

\begin{proof}
Let $\bf u_n=(u_{n,1},\ldots,u_{n,\ell})$ be given by $u_{n,i}(x) = \eps_n^{\frac1{p-1}}v_{n,i}(\eps_n x)$. Then $\bf u_n$ is a least energy solution to the system~\eqref{eq:system2} and, by Theorem~\ref{thm:limit for single core}, $u_{n,i}\to u_i$ strongly in $E$. The result follows from this fact arguing as in Corollary~\ref{concentration}.
\end{proof}
\medskip

\begin{proof}[Proof of Theorem~\ref{thm:coupled}]
 This follows immediately from Theorem~\ref{thm:limit for single core} and Corollary~\ref{concentration2}.
\end{proof}
\medskip

As for the single equation \cite[Theorem 1.5]{chs}, the following decay estimate holds true for the solutions of the limit system. Recall that $N\geq 3$.

\begin{proposition}\label{prop:decay}
Let $\bf u=(u_1,\ldots,u_\ell)$ be a nonnegative solution of the limit system~\eqref{eq:system3}. Then there exists $\kappa>0$ such that
\begin{align*}
 u_i(x) \le \kappa|x|^{2-N}\qquad \text{for a.e. \ }x\in \rn\smallsetminus B_1(0).
\end{align*}
\end{proposition}

\begin{proof}
Let $W(x):=\kappa|x|^{2-N}$ for $|x|\geq 1$ with $\kappa:=\max\limits_{\partial B_1(0)}u_i$, which is finite by Lemma~\ref{maincont}.
 Set $W_i:=W-u_{i}$. Note that $W_i\geq 0$ on $\partial B_1(0)$. Since $\bf u$ solves~\eqref{eq:system3}, $W_i$ solves
\begin{align*}
 -\Delta W_i = \mu_i |u_{i}|^{2p-2}u_{i}-\sum\limits_{\substack{j=1\\j\neq i}}^\ell\lambda_{ij}|u_{j}|^p|u_{i}|^{p-2}u_{i}\geq 0\quad \text{in \ }\rn\smallsetminus B_1(0),\qquad W_i\geq 0\quad \text{ on }\partial B_1(0).
\end{align*}
Testing this equation with $W_i^-:=\min\{W_i,0\}\leq 0$, we get that $0\leq \int_{\rn}|\nabla W_i^-|^2\leq 0$. Thus, $W_i\geq 0$ in $\rn \smallsetminus B_1(0)$, and therefore, $0\leq u_{i}(x)\leq \kappa |x|^{2-N}$ for all $x\in \rn\smallsetminus B_1(0)$.
\end{proof}

\section{Asymptotic segregation for the limit system} \label{segregation}

Next we prove Theorem~\ref{seg:thm}, which analyzes the limit as $\lambda_{ij}\to -\infty$ of nonnegative least energy solutions of the limit system~\eqref{eq:system3}.

\begin{proof}[Proof of Theorem~\ref{seg:thm}]
To highlight the role of $\lambda_{k}$, we write $\mathcal{J}_{0,k}$ and $\mathcal{N}_{0,k}$ for the functional $\mathcal{J}_0$ and the set $\mathcal{N}_0$ associated to the system~\eqref{eq:system3} with $\lambda_{ij}=\lambda_{k}$ for all $i\neq j$; see Section~\ref{sec:single core}. By assumption,
$$c_{0,k}:= \inf_{\mathcal{N}_{0,k}} \mathcal{J}_{0,k} =\mathcal{J}_{0,k}(\bf u_k)=\frac{p-1}{2p}\sum_{i=1}^\ell\|u_{k,i}\|^2,$$
where the norm $\|u_{k,i}\|$ is given in~\eqref{norm}.
We define
\begin{align*}
\mathcal{M}:=\{(v_1,\ldots,v_\ell)\in E^\ell:v_i\neq 0,\;\|v_i\|^2=\int_{\mathbb R^N}\mu_iQ_1(x)|v_i|^{2p}, \text{ and }v_iv_j=0\text{ a.e. in }\mathbb{R}^N \text{ if }i\neq j\}.
\end{align*}
Then, $\mathcal{M}\subset\mathcal{N}_{0,k}$ for all $k\in\mathbb{N}$ and, as a consequence,
$$0<c_{0,k}\leq c_*:=\inf\Big\{\frac{p-1}{2p}\sum_{i=1}^\ell\|v_i\|^2:(v_1,\ldots,v_\ell)\in\mathcal{M}\Big\}<\infty.$$
This shows that $(u_{k,i})$ is bounded in $D^{1,2}(\mathbb{R}^N)$ and arguing as in~\eqref{2pbd} we see that $(u_{k,i})$ is also bounded in $L^{2p}(\mathbb{R}^N)$. So, after passing to a subsequence, $u_{k,i} \rightharpoonup u_{\infty,i}$ weakly in $E$, $u_{k,i} \to u_{\infty,i}$ strongly in $L_{loc}^{2p}(\mathbb{R}^N)$ and $u_{k,i} \to u_{\infty,i}$ a.e. in $\mathbb{R}^N$, for each $i=1,\ldots,\ell$. Hence, $u_{\infty,i} \geq 0$. Moreover, as $\partial_i\mathcal{J}_{0,k}(\bf u_k)[u_{k,i}]=0$, we have that, for each $j\neq i$,
\begin{align*}
0&
\leq |\lambda_{k}| \int_{\mathbb{R}^N}|u_{k,j}|^{p}|u_{k,i}|^{p}
\leq \mu_i\int_{|x|<1} |u_{k,i}|^{2p}\leq C\qquad\text{for all \ }k\in\n.
\end{align*}
As $|\lambda_{k}|\to\infty$, using Fatou's lemma we obtain
$$0 \leq \int_{\mathbb{R}^N}|u_{\infty,j}|^{p}|u_{\infty,i}|^{p} \leq \liminf_{k \to \infty} \int_{\mathbb{R}^N}|u_{k,j}|^{p}|u_{k,i}|^{p} = 0.$$
Therefore, $u_{\infty,j} u_{\infty,i} = 0$ a.e. in $\mathbb{R}^N$. On the other hand, arguing as in~\eqref{sobolev}-\eqref{lbd:sob}, there is a constant $d_0>0$ such that
$$0<d_0 \leq \|u_{k,i}\|^2 \leq  \int_{|x|<1} |u_{k,i}|^{2p}\qquad\text{for all \ }k\in\mathbb{N},\qquad i=1,\ldots,\ell.$$
As $u_{k,i} \to u_{\infty,i}$ in $L_{loc}^{2p}(\mathbb{R}^N)$ this implies that $u_{\infty,i}\neq 0$. Furthermore, using Fatou's Lemma
we obtain
\begin{align} \label{eq:subnehari}
\|u_{\infty,i}\|^2 & \leq \liminf_{k\to\infty}\|u_{k,i}\|^2 \leq \limsup_{k\to\infty}\|u_{k,i}\|^2 \leq \limsup_{k\to\infty}\Big(\irn\mu_i Q_1(x)|u_{k,i}|^{2p}\Big)\nonumber \\
&\leq \lim_{k\to\infty}\int_{|x|<1}\mu_i|u_{k,i}|^{2p}-\liminf_{k\to\infty}\int_{|x|\geq 1}\mu_i|u_{k,i}|^{2p} \nonumber \\
&\leq \int_{|x|<1}\mu_i|u_{\infty,i}|^{2p} - \int_{|x|\geq 1}\mu_i|u_{\infty,i}|^{2p} =\irn \mu_i Q_1(x)|u_{\infty,i}|^{2p},
\end{align}
for all $i=1,\ldots,\ell$. Hence, there is a unique $(t_1,\ldots,t_\ell)\in(0,1]^\ell$ such that $(t_1u_{\infty,1},\ldots,t_\ell u_{\infty,\ell})\in \mathcal{M}$. It follows that
\begin{align*}
c_* &\leq \frac{p-1}{2p}\sum_{i=1}^\ell\|t_iu_{\infty,i}\|^2 \leq \frac{p-1}{2p}\sum_{i=1}^\ell\|u_{\infty,i}\|^2\leq \frac{p-1}{2p}\liminf_{k\to\infty}\sum_{i=1}^\ell\|u_{k,i}\|^2=\liminf_{k\to\infty} c_{0,k} \leq c_*.
\end{align*}
Hence, $u_{k,i} \to u_{\infty,i}$ strongly in $D^{1,2}(\mathbb{R}^N)$ and  $t_i=1$, yielding
\begin{equation}\label{eq:limit}
\|u_{\infty,i}\|^2 = \int_{\rn} \mu_iQ_1(x)|u_{\infty,i}|^{2p}\qquad\text{and}\qquad\frac{p-1}{2p}\sum_{i=1}^\ell\|u_{\infty,i}\|^2=c^*.
\end{equation}
Combining the first identity with~\eqref{eq:subnehari}, since $u_{k,i} \to u_{\infty,i}$ in $L_{loc}^{2p}(\mathbb{R}^N)$, we get that
$u_{k,i} \to u_{\infty,i}$ strongly in $L^{2p}(\mathbb{R}^N)$, and then $u_{k,i} \to u_{\infty,i}$ strongly in $E$.

By Lemma~\ref{maincont}, we have that $(u_{k,i})$ is uniformly bounded in $L^{\infty}(\rn)$. Then, by \cite[Theorem 1.2]{sttz}, for every bounded domain $U$ in $\rn$ and every $\alpha\in(0,1)$ there is $C=C(U,\alpha)>0$ such that
\begin{align*}
 \|u_{k,i}\|_{\cC^{0,\alpha}(U)}<C,
\end{align*}
where $\| \cdot \|_{\cC^{0,\alpha}(U)}$ denotes the usual norm in the space of $\alpha$-Hölder-continuous functions. This implies that $u_{\infty,i}$ is continuous.

Let $\o_i:=\{x\in\rn:u_{\infty,i}(x)>0\}$. This is an open set. As $u_{\infty,i}u_{\infty,j}=0$ we have that $\o_i\cap\o_j=\emptyset$ if $i\neq j$. The first identity in~\eqref{eq:limit} shows that $u_{\infty,i}$ belongs to the Nehari manifold
\begin{equation} \label{neh2}
\cN_{\o_i}:=\Big\{w\in E_i:w\neq 0, \ \|w\|^2=\int_{\o_i}Q_1(x)|w|^{2p}\Big\},
\end{equation}
where $E_i:=D_0^{1,2}(\Omega_i)\cap L^{2p}(\Omega_i)$. From the second identity in~\eqref{eq:limit} one easily derives that
$$\frac{p-1}{2p}\|u_{\infty,i}\|^2=\inf_{w\in\cN_{\o_i}}\frac{p-1}{2p}\|w\|^2\qquad\text{for every \ }i=1,\ldots,\ell.$$
This shows that $u_{\infty,i}$ is a least energy solution of~\eqref{eq:omega_i}, as claimed. Finally, \eqref{neh2} implies $\o_i\cap B_1(0)\ne\emptyset$.
\end{proof}

\section{Concentration at higher energy levels} \label{sec:conc}

In Corollaries~\ref{concentration} and~\ref{concentration2} we have shown that nonnegative least energy solutions to~\eqref{eq:system} concentrate as $\eps_n\to 0$. This was an easy consequence of the existence of limit profiles. Here we shall show that concentration occurs for any sequence $(v_{n,i})$ of solutions as $\eps_n\to 0$. The proof is a modification of the argument of \cite[Theorem 2.6]{as}. Let the norm of $v\in H^1(\rn)$ be given by
\[
\|v\|_1^2 := \irn(|\nabla v|^2+v^2).
\] 

\begin{lemma} \label{infty}
Let $\eps_n\to 0$ and let $\bf v_n = (v_{n,1},\ldots,v_{n,\ell})$ be a solution to~\eqref{eq:system} with $\eps=\eps_n$. Then $\|v_{n,i}\|_1\to\infty$ for $i=1,\ldots,\ell$.
\end{lemma}

\begin{proof}
Multiplying the equation in~\eqref{eq:system} by $v_{n,i}$ we obtain
\begin{equation} \label{equality}
\|v_{n,i}\|_1^2 = \mu_i\irn Q_{\varepsilon_n}(x-y_i)|v_{n,i}|^{2p} + \sum\limits_{\substack{j=1 \\ j\neq i}}^\ell\irn\lambda_{ij}|v_{n,j}|^p|v_{n,i}|^{p}.
\end{equation}
So, by the Sobolev inequality and since $\lambda_{ij}<0$,
\begin{equation} \label{lpnorm}
\|v_{n,i}\|_1^2 \le \mu_i\irn|v_{n,i}|^{2p} \le C\|v_{n,i}\|_1^{2p}
\end{equation}
for some constant $C>0$ independent of $n$. Hence, $\|v_{n,i}\|_1\ge \alpha$ for some $\alpha>0$.

Now we argue by contradiction. Assume that the sequence $(v_{n,i})$ is bounded in $H^1(\rn)$ and let $q=2^*$. By the H\"older inequality and as $(v_{n,i})$ is bounded in $L^q(\rn)$,
\[
0<\alpha^2\leq\|v_{n,i}\|_1^2 \le \mu_i\int_{|x-y_i|<\eps_n}|v_{n,i}|^{2p} \le \mu_i\,(\text{meas}\{|x-y_i|<\eps_n\})^\frac{1}{r}\Big(\irn|v_{n,i}|^q\Big)^\frac{2p}q \to 0,
\]
where  $r$ satisfies $\frac1r+\frac{2p}q=1$. This is a contradiction.
\end{proof}

\begin{lemma} \label{weakto0}
Let $\eps_n$ and $\bf v_n$ be as in Lemma \ref{infty} and let $z_{n,i} := \frac{v_{n,i}}{\|v_{n,i}\|_1}$. Then $z_{n,i}\rh 0$ weakly in $H^1(\rn)$.
\end{lemma}

\begin{proof}
Passing to a subsequence, $z_{n,i}\rh z_i$ weakly in $H^1(\rn)$, $z_{n,i}\to z_i$ strongly in $L^{2p}_{loc}(\rn)$ and a.e. in $\rn$. Dividing~\eqref{equality} by $\|v_{n,i}\|_1^2$ we get
\begin{equation} \label{equals1}
1 = \|z_{n,i}\|_1^2 = \mu_i\|v_{n,i}\|_1^{2p-2}\Big(\irn Q_{\eps_n}(x-y_i)|z_{n,i}|^{2p} + \|v_{n,i}\|_1^{-2p}\sum\limits_{j\neq i}\irn\lambda_{ij}|v_{n,j}|^p|v_{n,i}|^p\Big).
\end{equation}
Since $\|v_{n,i}\|_1\to\infty$, the above equality implies that
$$\irn Q_{\eps_n}(x-y_i)|z_{n,i}|^{2p} + \|v_{n,i}\|_1^{-2p}\sum\limits_{j\neq i}\irn\lambda_{ij}|v_{n,j}|^p|v_{n,i}|^p\to 0.$$ 
If $z_i\ne 0$, we can choose $\eps_0$ such that $\int_{|x-y_i|\ge \eps_0}|z_i|^{2p} > 0$. Then, for all $n$ large enough, using Fatou's lemma we obtain
\begin{align*}
0 & = \lim_{n\to\infty}\Big(\irn Q_{\eps_n}(x-y_i)|z_{n,i}|^{2p} + \|v_{n,i}\|_1^{-2p}\sum\limits_{j\neq i}\irn\lambda_{ij}|v_{n,j}|^p|v_{n,i}|^p\Big) \\
&\leq\limsup_{n\to\infty} \Big(\int_{|x-y_i|<\eps_n} |z_{n,i}|^{2p} - \int_{|x-y_i|\ge\eps_n}|z_{n,i}|^{2p}\Big) \le - \int_{|x-y_i|\ge\eps_0}|z_i|^{2p} < 0, 
\end{align*}
a contradiction. Hence $z_i$ must be 0.
\end{proof}
\medskip

\begin{proof}[Proof of Theorem~\ref{thm:concentr}]
Fix $\delta>0$ and let $\chi\in \cC^{\infty}(\rn)$ be such that $\chi(x)=0$ for $|x-y_i|<\frac{\delta}2$, $\chi(x) = 1$ for $|x-y_i|>\delta$. Multiplying the equation in~\eqref{eq:system} by $\chi v_{n,i}$ and using $\lambda_{ij}<0$ gives
\[
\irn(\nabla v_{n,i}\cdot\nabla(\chi v_{n,i}) + \chi v_{n,i}^2)  \le \mu_i\irn \chi Q_{\varepsilon_n}(x-y_i)|v_{n,i}|^{2p},
\]
or, equivalently,
\[
\irn\chi(|\nabla v_{n,i}|^2+v_{n,i}^2) - \mu_i\irn \chi Q_{\varepsilon_n}(x-y_i)|v_{n,i}|^{2p} \le -\irn v_{n,i}\nabla\chi\cdot \nabla v_{n,i}.
\]
As $Q_{\eps_n} = -1$ on the support of $\chi$ if $n$ is large, for such $n$ we have
\begin{align} \label{geeps}
\int_{|x-y_i|>\delta}(|\nabla v_{n,i}|^2+v_{n,i}^2) + \mu_i \int_{|x-y_i|>\delta}|v_{n,i}|^{2p} & \le \irn\chi(|\nabla v_{n,i}|^2+v_{n,i}^2) + \mu_i \irn\chi|v_{n,i}|^{2p} \nonumber \\
& \le C\int_{\delta/2<|x-y_i|<\delta}|v_{n,i}|\,|\nabla v_{n,i}|, 
\end{align}
where $C= \max_{\rn}(|\nabla\chi|)$.
Since $z_{n,i}\to 0$ in $L^2_{loc}(\rn)$ according to Lemma~\ref{weakto0}, it follows, using H\"older's inequality, that
\[
\int_{\delta/2<|x-y_i|<\delta}|z_{n,i}|\,|\nabla z_{n,i}| \to 0 .
\]
Hence,~\eqref{geeps} implies that
\[
\lim_{n\to\infty}\Big(\int_{|x-y_i|>\delta}(|\nabla z_{n,i}|^2+z_{n,i}^2) + \mu_i \|v_{n,i}\|_1^{2p-2}\int_{|x-y_i|>\delta}|z_{n,i}|^{2p}\Big) = 0.  
\]
So,
\begin{equation} \label{lim0}
\lim_{n\to\infty} \int_{|x-y_i|>\delta}(|\nabla z_{n,i}|^2+z_{n,i}^2) = 0 \quad \text{and} \quad \lim_{n\to\infty} \|v_{n,i}\|_1^{2p-2}\int_{|x-y_i|>\delta}|z_{n,i}|^{2p} = 0.
\end{equation}
As $\|z_{n,i}\|_1=1$, using the first limit above we obtain that
\[
1 = \lim_{n\to\infty}\frac{\irn(|\nabla z_{n,i}|^2+z_{n,i}^2)}{\irn(|\nabla z_{n,i}|^2+z_{n,i}^2)} = \lim_{n\to\infty}\frac{\int_{|x-y_i|\le\delta}(|\nabla z_{n,i}|^2+z_{n,i}^2)}{\irn(|\nabla z_{n,i}|^2+z_{n,i}^2)} = \lim_{n\to\infty} \frac{\int_{|x-y_i|\le\delta}(|\nabla v_{n,i}|^2+v_{n,i}^2)}{\irn(|\nabla v_{n,i}|^2+v_{n,i}^2)}.
\]
Since 
\[
\irn Q_{\eps_n}(x-y_i)|z_{n,i}|^{2p} \le \irn |z_{n,i}|^{2p},
\]
it follows from~\eqref{equals1} that $\|v_{n,i}\|_1^{2p-2}\irn |z_{n,i}|^{2p}$ is bounded away from 0. Hence, the second limit in~\eqref{lim0} implies that
\begin{align}\label{aux:eq}
1 = \lim_{n\to\infty}\frac{\|v_{n,i}\|_1^{2p-2}\irn |z_{n,i}|^{2p}}{\|v_{n,i}\|_1^{2p-2}\irn |z_{n,i}|^{2p}} = \lim_{n\to\infty}\frac{\|v_{n,i}\|_1^{2p-2}\int_{|x-y_i|\le\delta} |z_{n,i}|^{2p}}{\|v_{n,i}\|_1^{2p-2}\irn |z_{n,i}|^{2p}} =  \lim_{n\to\infty}\frac{\int_{|x-y_i|\le\delta} |v_{n,i}|^{2p}}{\irn |v_{n,i}|^{2p}}.
\end{align}
This completes the proof of Theorem~\ref{thm:concentr}.
\end{proof}

\begin{remark}
\emph{
By~\eqref{lpnorm} and Lemma~\ref{infty}, we have that $\irn |v_{n,i}|^{2p}\to\infty$ as $n\to\infty$. Then, by \eqref{aux:eq}, $\int_{|x-y_i|\le\delta} |v_{n,i}|^{2p}\to \infty$ and, therefore, $\sup_{B_\delta(y_i)}|v_{n,i}|\to \infty$ for each $\delta>0$.
}
\end{remark}

\begin{remark}\label{decay}
\emph{Let $\eps_n\to 0$, let $\bf v_n = (v_{n,1},\ldots,v_{n,\ell})$ be a nonnegative solution of the system~\eqref{eq:system} with $\eps=\eps_n$, set $w_{n,i}(x):=\eps_n^\frac{1}{p-1}v_{n,i}(\eps_nx+y_i)$ and assume that $w_{n,i}\to w_{\infty,i}$ in $L^{2p}(\rn)$ for some $w_{\infty,i}\in L^{2p}(\rn)$ for $i=1,\ldots, \ell$.  If $2p\in(\frac{2N-2}{N-2},2^*)$, then
\begin{center}
$v_{n,i}\to 0$ in $L^{2p}(\rn\smallsetminus B_\delta(y_i))$ for all $\delta>0$.
\end{center}
 Indeed, let $C>0$ denote possibly different constants independent of $n$. By Lemma~\ref{maincont} and a simple comparison argument (as in Proposition~\ref{prop:decay}), $w_{n,i}(x)\leq C|x|^{2-N}$ for $x\in\rn$ and $n\in \mathbb N$. Then,
$$\int_{|x|>\frac{\delta}{\eps_n}}|w_{n,i}|^{2p}\leq C\int_{\frac{\delta}{\eps_n}}^\infty r^{(2-N)2p+N-1} \mathrm{d} r=C\eps_n^{(N-2)2p-N},$$
where we used that $(N-2)2p-N>0$, because $2p>\frac{2N-2}{N-2}>\frac N{N-2}$ for $N\geq 3$. Then,
$$\int_{|x-y_i|>\delta}|v_{n,i}|^{2p}=\eps_n^{N-\frac{2p}{p-1}}\int_{|x|>\frac{\delta}{\eps_n}}|w_{n,i}|^{2p}\leq C\eps_n^{((N-2)(p-1)-1)\frac{2p}{p-1}}\to 0,$$
because $(N-2)(p-1)-1>0$ if $p>\frac{N-1}{N-2}$.}
\end{remark}

\appendix

\section{A uniform bound for the limit system}

In this appendix we apply well known regularity arguments to show the uniform boundedness property used in the proof of Theorem~\ref{seg:thm} and in Remark \ref{decay}.

We write $B_1:=B_1(0)$ and $| \cdot |_{r;B_1}$ for the norm in $L^r(B_1)$, $r\in[1,\infty]$.

\begin{lemma}\label{maincont}
Let $N\geq 3$ and, for each $i,j=1,\ldots,\ell$, $i\neq j$, let $(\lambda_{ij,k})$ be a sequence of negative numbers, $(\eps_k)$ a sequence of real numbers, and $\bf u_{k}=(u_{k,1},\ldots,u_{k,\ell})$ be a nonnegative solution of the system
\begin{equation}\label{eq:ap}
-\Delta u_{k,i} +\eps_k^2 u_{k,i}= \mu_iQ_1(x)|u_{k,i}|^{2p-2}u_{k,i} + \sum\limits_{\substack{j=1 \\ j\neq i}}^\ell\lambda_{ij,k}|u_{k,i}|^p|u_{k,i}|^{p-2}u_{k,i},\quad u_{k,i}\neq 0,\quad i=1,\ldots,\ell.
\end{equation}
Here, $u_{k,i}\in E$ if $\eps_k=0$ and $u_{k,i}\in H^1(\rn)$ if $\eps^2_k>0$. Assume that $u_{k,i} \to u_{\infty,i}$ strongly in $L^{2p}(\rn)$ for every $i=1,\ldots,\ell$. Then $(u_{k,i})$ is uniformly bounded in $L^{\infty}(\rn)$ for every $i=1,\ldots,\ell$.
\end{lemma}

\begin{proof}
We adapt the arguments in \cite[Lemma B.3]{s}. Let $s\geq 0$ and assume that $u_{k,i}\in L^{2(s+1)}(B_1)$ for every $k\in\mathbb{N}$. Fix $M>0$ and define $\varphi_{k,i}:=u_{k,i}\min\{|u_{k,i}|^{2s},M^{2}\}$. Then,
\begin{equation*}
\nabla \varphi_{k,i} =\min \{|u_{k,i}|^{2s},M^{2} \}\nabla u_{k,i}+2s|u_{k,i}|^{2s}(\nabla u_{k,i})1_{A},
\end{equation*}
where $1_{A}$ is the characteristic function for the set $A:=\{x\in\rn: |u_{k,i}(x)|^{s}< M \}$. Since $\bf u_k$ solves~\eqref{eq:ap} and $u_{k,i}\varphi_{k,i}\geq 0$, we have that
\begin{align*}
\int_{\rn}\nabla u_{k,i}\cdot \nabla \varphi_{k,i}
&=  \int_{\rn}\mu_iQ_1(x)|u_{k,i}|^{2p-2}u_{k,i}\varphi_{k,i}+\sum\limits_{\substack{j=1\\j\neq i}}^\ell\lambda_{ij,k}\int_{\rn}|u_{j,k}|^p|u_{k,i}|^{p-2}u_{k,i}\varphi_{k,i}
-\eps^2_k\irn u_{k,i}\vp_{k,i}\\
&\leq \int_{B_1}\mu_i|u_{k,i}|^{2p-2}u_{k,i}\varphi_{k,i}.
\end{align*}
On the other hand,
\begin{equation*}
\begin{split}
\nabla u_{k,i}\cdot \nabla \varphi_{k,i}= \min \{|u_{k,i}|^{2s},M^{2} \}|\nabla u_{k,i} |^{2}+2s|u_{k,i}|^{2s}|\nabla u_{k,i}|^{2}1_{A}\geq  \min \{|u_{k,i}|^{2s},M^{2} \}|\nabla u_{k,i} |^{2}.
\end{split}
\end{equation*}
Since the embedding $H^1(B_1)\hookrightarrow L^{2p}(B_1)$ is continuous, there is $C=C(N,p,|B_1|)>0$ such that
\begin{align}
&\Big(\int_{B_1} \left| \min \{|u_{k,i}|^{s},M \} u_{k,i}\right|^{2p} \Big)^{\frac{2}{2p}}\leq  C\int_{B_1} \left| \nabla\left( \min \{|u_{k,i}|^{s},M \} u_{k,i}\right) \right|^{2}
+C\int_{B_1} \left| \min \{|u_{k,i}|^{s},M \} u_{k,i}\right|^{2} 
\notag\\
&\le C\int_{\rn} \left| \left( \min \{|u_{k,i}|^{s},M \} \nabla u_{k,i} +s|u_{k,i}|^{s}(\nabla u_{k,i})1_{A} \right) \right|^{2}
+C\int_{B_1} \left(\min \{|u_{k,i}|^{2s},M^2\} u_{k,i}\right)u_{k,i}
\notag\\
&\leq  2C\left( \int_{\rn}  \min \{|u_{k,i}|^{2s},M^{2} \} \left|\nabla u_{k,i}\right|^{2} + s^{2}\int_{\rn} |u_{k,i}|^{2s}1_{A}\left| \nabla u_{k,i} \right|^{2}\right)
+C\int_{B_1} u_{k,i}\varphi_{k,i}
\notag\\
&\leq 2C(1+s^{2})\int_{\rn}  \min \{|u_{k,i}|^{2s},M^{2} \} \left|\nabla u_{k,i}\right|^{2}+C\int_{B_1} u_{k,i}\varphi_{k,i}\notag\\
&\leq  2C(1+s^{2})\int_{\rn}\nabla u_{k,i}\cdot \nabla\varphi_{k,i}
+C\int_{B_1} u_{k,i}\varphi_{k,i}\notag\\
&\leq  2C(1+s^{2})\mu_i\int_{B_1}|u_{k,i}|^{2p-2}u_{k,i}\varphi_{k,i}
+C\int_{B_1} u_{k,i}\varphi_{k,i}.  \label{1}
\end{align}
Let $K>0$ and set $D_{i,K}:=\{x\in B_1: |u_{\infty,i}|^{2p-2}\geq K/\mu_i \}$ and $D_{i,K}^{c}:=\{x\in B_1: |u_{\infty,i}|^{2p-2}< K/\mu_i \}$. Since $u_{k,i}\varphi_{k,i}\in L^p(B_1)$ and $u_{k,i}\varphi_{k,i}\geq 0$, using Hölder's inequality, we obtain
\begin{align}
&\int_{B_1}|u_{k,i}|^{2p-2}u_{k,i}\varphi_{k,i}\notag \\
& = \int_{B_1}\left(|u_{k,i}|^{2p-2}-|u_{\infty,i}|^{2p-2} \right) u_{k,i}\varphi _{k,i} +\int_{B_1}|u_{\infty,i}|^{2p-2}  u_{k,i}\varphi _{k,i} \notag\\
& = \int_{B_1}\left(|u_{k,i}|^{2p-2}-|u_{\infty,i}|^{2p-2}\right) u_{k,i}\varphi _{k,i} +\int_{D_{i,K}}|u_{\infty,i}|^{2p-2}  u_{k,i}\varphi _{k,i}+\int_{D_{i,K}^{c}}|u_{\infty,i}|^{2p-2}  u_{k,i}\varphi _{k,i} \notag\\
&\leq \Big( \int_{B_1}\left| |u_{k,i}|^{2p-2}-|u_{\infty,i}|^{2p-2}\right|^{\frac{p}{p-1}}\Big)^{\frac{p-1}{p}}\left| u_{k,i}\varphi _{k,i}\right|_{p;B_1}+\Big(\int_{D_{i,K}} |u_{\infty,i}|^{2p} \Big)^{\frac{p-1}{p}} \left| u_{k,i}\varphi _{k,i}\right|_{p;B_1} +\frac K{\mu_i}\int_{B_1}  u_{k,i}\varphi _{k,i}.
\label{2}
\end{align}
Since $u_{\infty,i}\in L^{2p}(\rn)$, we may fix $K$ large enough so that
\begin{equation}\label{3}
	\Big(\int_{D_{i,K}} |u_{\infty,i}|^{2p} \Big)^{\frac{p-1}{p}} <\frac{1}{8C\mu_i(1+s^{2})}.
\end{equation}
Moreover, since  $u_{k,i}\rightarrow u_{\infty,i}$ strongly in $ L^{2p}(\rn)$, there exist $D>0$ and $k_{0}\in \mathbb N$ such that
\begin{equation}\label{4}
	\left( \int_{B_1}\left| |u_{k,i}|^{2p-2}-|u_{\infty,i}|^{2p-2}\right|^{\frac{p}{p-1}}\right)^{\frac{p-1}{p}}\leq D	 \left( \int_{B_1}\left| u_{k,i}-u_{\infty,i}\right|^{2p}\right)^{\frac{p-1}{p}} \leq \frac{1}{8C\mu_i(1+s^{2})}
\end{equation}
for any $k>k_{0}$. Thus, using~\eqref{1},~\eqref{2},~\eqref{3} and~\eqref{4}, we have
\begin{align*}
\left| u_{k,i}\varphi _{k,i}\right|_{p;B_1}
&=\left(\int_{B_1}\left| u_{k,i}\varphi _{k,i}\right|^{p}\right)^\frac{1}{p}
=\Big(\int_{B_1} \left| \min \{|u_{k,i}|^{s},M \} u_{k,i}\right|^{2p} \Big)^{\frac{2}{2p}}\\
&\leq \frac{1}{4}\left| u_{k,i}\varphi _{k,i}\right|_{p;B_1}+ \frac{1}{4}\left| u_{k,i}\varphi _{k,i}\right|_{p;B_1} + 2C(1+s^{2})K\int_{B_1}  u_{k,i}\varphi _{k,i}
+C\int_{B_1} u_{k,i}\varphi_{k,i}\\
&\leq \frac{1}{2}\left| u_{k,i}\varphi _{k,i}\right|_{p;B_1}
+ 3C(1+s^{2})K\int_{B_1}  u_{k,i}\varphi _{k,i}.
\end{align*}
Hence,
\begin{align*}
\frac{1}{2}\left(\int_{B_1}|u_{k,i}\min\{|u_{k,i}|^{2s},M^{2}\}u_{k,i}|^{p}\right)^{\frac{2}{2p}}
&=\frac{1}{2}\left| u_{k,i}\varphi _{k,i}\right|_{p;B_1}\leq 3K C(1+s^{2})\int_{B_1}u_{k,i}\varphi _{k,i} \\
&=3K C(1+s^{2})\int_{B_1}\min\{|u_{k,i}|^{2s},M^{2}\}u_{k,i}^{2}.
\end{align*}
Since we assumed that $u_{k,i}\in L^{2(s+1)}(B_1)$, we can pass to the limit as $M\to\infty$ which gives
\begin{align}
\int_{B_1}  |u_{k,i}|^{2p(s+1)}
&\leq \left((6K C)(1+s^{2})\left(\int_{B_1} |u_{k,i}|^{2(s+1)}\right)\right)^p \nonumber \\
&\leq (6 K C)^p(1+s)^{2p}\left(\int_{B_1} |u_{k,i}|^{2(s+1)}\right)^p.\label{s1}
\end{align}

Now, to obtain a uniform bound in $L^\infty(B_1)$, we argue as in \cite[Lemma 3.2]{tt}. First, let $C_1$ be such that $\int_{\rn} |u_{k,i}|^{2p} \leq C_1$ for all $k\in\n$, and let $C_0:=6 K C$. Then, setting $s:=p-1$, estimate~\eqref{s1} yields
\begin{equation*}
\int_{B_1}  |u_{k,i}|^{2p^2}
\leq C_0^p p^{2p}C_1^p=:C_2.
\end{equation*}
Next, setting $s:=p^2-1$, estimate~\eqref{s1} yields that
\begin{equation*}
\int_{B_1}  |u_{k,i}|^{2p^3}
\leq C_0^p p^{4p} C_2^p
=C_0^{p+p^2} p^{2(p^2+2p)} C_1^{p^2}
=:C_3.
\end{equation*}
Setting $s_m:=p^m-1$ and  iterating this procedure, for each $m\in \mathbb N$, we obtain
\begin{equation*}
\int_{B_1}  |u_{k,i}|^{2p^m}
\leq C_0^{\sum_{n=1}^m p^n} p^{2\sum_{n=1}^m (m-n+1) p^{n}} C_1^{p^m}=:C_m.
\end{equation*}
Let $q:= p^{-1}$. Then
\[
\frac{1}{p^m}\sum_{n=1}^m p^n = \sum_{n=1}^m q^{m-n} = \frac{1-q^m}{1-q} \to \frac1{1-q} = \frac p{p-1}\qquad \text{ as }m\to \infty 
\]
and 
\begin{align*}
 \frac{1}{p^m}\sum_{n=1}^m (m-n+1) p^{n} & = \sum_{n=1}^m (m-n+1) q^{m-n} = \frac d{dq}\sum_{n=1}^m q^{m-n+1} = \frac d{dq}\Big(q\frac{1-q^m}{1-q}\Big) \to \frac d{dq}\Big(\frac q{1-q}\Big) \\
 & = \frac1{(1-q)^2} = \frac{p^2}{(p-1)^2} \qquad \text{as } m\to\infty.
\end{align*}
Hence,
\begin{align*}
|u_{k,i}|_{\infty;B_1}
=\lim_{m\to \infty}
|u_{k,i}|_{2p^m;B_1} \leq \lim_{m\to\infty} C_m^{\frac{1}{2p_m}}
= C_0^{\frac{p}{2(p-1)}} p^{\frac{p^2}{(p-1)^2}} C_1^{\frac{1}{2}}=:\kappa\qquad\text{for every \ }k\in\n.
\end{align*}
Finally, we extend this bound to the rest of $\rn$. Set $\phi:=\max\{u_{k,i}-\kappa,0\}\geq 0$. Note that $\nabla u_{i,k}\cdot\nabla \phi=1_{\{u_{i,k}>\kappa\}}|\nabla u_{i,k}|^2=|\nabla \phi|^2$. Then, since $\bf u_k$ satisfies~\eqref{eq:ap} and $\phi(x)=0$ for a.e. $x\in B_1$,
\begin{align*}
0\leq \int_{\rn}|\nabla \phi|^2
=\int_{\rn}\nabla u_{k,i}\cdot\nabla \phi
=\int_{\rn\smallsetminus B_1}\Bigg(-\mu_1 u_{k,i}^{2p-1}-\eps_k^2u_{k,i}+\sum\limits_{\substack{j=1\\j\neq i}}^\ell\lambda_{ij,k}u_{k,j}^pu_{k,i}^{p-1}
\Bigg)\phi\leq 0.
\end{align*}
This implies that $\phi\equiv 0$ and therefore $u_{k,i}\leq \kappa$ in $\rn$ as claimed.
\end{proof}

\begin{remark}
\emph{
Note that the last part in the proof of Lemma \ref{maincont} shows that the supremum of nonnegative solutions to systems like \eqref{eq:ap} is always achieved at the region of attraction (where the coefficient $Q_1$ is positive).
}
\end{remark}

\noindent\textbf{Statements and Declarations}

Conflict of Interest: The authors declare that they have no conflict of interest.

\bigskip

\begin{flushleft}
\textbf{Mónica Clapp}\\
Instituto de Matemáticas\\
Universidad Nacional Autónoma de México \\
Campus Juriquilla\\
76230 Querétaro, Qro., Mexico\\
\texttt{monica.clapp@im.unam.mx}
\medskip

\textbf{Alberto Saldaña}\\
Instituto de Matemáticas\\
Universidad Nacional Autónoma de México \\
Circuito Exterior, Ciudad Universitaria\\
04510 Coyoacán, Ciudad de México, Mexico\\
\texttt{alberto.saldana@im.unam.mx}
\medskip

\textbf{Andrzej Szulkin}\\
Department of Mathematics\\
Stockholm University\\
106 91 Stockholm, Sweden\\
\texttt{andrzejs@math.su.se}
\end{flushleft}

\end{document}